\documentclass[12pt]{article}
\usepackage[T1]{fontenc}      
\RequirePackage[applemac]{inputenc} 

\scrollmode
\usepackage{color}
\usepackage{amsfonts}
\usepackage{amsmath}
\usepackage{amssymb}
\usepackage{graphicx}
\usepackage{subfigure}
\usepackage{float}
\usepackage{parskip}
\RequirePackage[colorlinks, hyperindex, plainpages=false]{hyperref}
\usepackage{hyperref}
\usepackage{amsmath}

\def\P{{\mathbb P}}
\def\E{{\mathbb E}}
\def\Z{{\mathbb Z}}

\def\R{{\mathbb R}}

\def\1{{\mathbf 1}}

\newtheorem{theorem}{Theorem}[section]
\newtheorem{lemma}[theorem]{Lemma}
\newtheorem{corollary}[theorem]{Corollary}
\newtheorem{proposition}[theorem]{Proposition}

\newtheorem{remark}[theorem]{Remark}

\newenvironment{proof}[1][Proof.]{\textbf{#1} }{\hfill $\blacksquare$}
\def\beq{\begin{equation}}
\def\eeq{\end{equation}}
\newcommand{\bei}{\begin{itemize}}
\newcommand{\eei}{\end{itemize}}
\newcommand{\ben}{\begin{enumerate}}
\newcommand{\een}{\end{enumerate}}
\newcommand{\beqn}{\begin{eqnarray}}
\newcommand{\beqnn}{\begin{eqnarray*}}
\newcommand{\eeqn}{\end{eqnarray}}
\newcommand{\eeqnn}{\end{eqnarray*}}
\newcommand{\brm}{\begin{rmk}}
\newcommand{\erm}{\end{rmk}}

\begin{document}
\title{Convergence to a Continuous State Branching Process with jumps and Height Process.}
\author{
I.~Dram\'{e}  \footnote{  \scriptsize{Aix-Marseille Université, CNRS,  Centrale Marseille, I2M, UMR 7373, 13453 Marseille, France. ibrahima.drame@univ-amu.fr
}}
\and
E.~Pardoux \setcounter{footnote}{6}\footnote{ \scriptsize {Aix-Marseille Université, CNRS,  Centrale Marseille, I2M, UMR 7373, 13453 Marseille, France. etienne.pardoux@univ-amu.fr}}
}

\maketitle

\begin{abstract}
In this work, we study asymptotics of the genealogy of Galton-Watson processes. Thus we consider a offspring distribution such that the rescaled Galton-Watson processes converges to a continuous state branching process (CSBP) with jumps. After we show that the rescaled height (or exploration) process of the corresponding Galton-Watson family tree, converges in a functional sense, to the continuous height process that Le Gall and Le Jan  introduced \cite{le1998branching}. 
\end{abstract}

\vskip 3mm
\noindent{\textbf{Keywords}: } Continuous-State Branching Processes; Scaling Limit; Galton-Watson Processes; Lévy Processes; Local time; Height Process;
\vskip 3mm
\section{Introduction}
Continuous state branching processes (or CSBP in short) are the analogues of Galton-Watson (G-W) processes in continuous time and continuous state space. Such classes of processes have been introduced by Jirina \cite{jivrina1958stochastic} and studied by many authors included Grey \cite{grey1974asymptotic}, Lamperti \cite{lamperti1967continuous},  to name but a few. These processes are the only possible weak limits that can be obtained from sequences of rescaled G-W processes (see  \cite{lamperti1967limit}). In Li (\cite{li2010measure} ,\cite{li2012continuous}), it was shown that the CSBP arises naturally as the scaling limit of a sequence of discrete G-W branching processes. 

However, If a scaling limit of discrete-time G-W processes converges to a CSBP, then it has been shown in \cite{duquesne2002random}, Chapter 2, that the genealogical structure of the G-W processes converges too. More precisely, the corresponding rescaled sequences of discrete height process, converges to the height process in continuous time that has been introduced by Le Gall and Le Jan in \cite{le1998branching}.

In this work, we are interested in continuous versions of this correspondence. Indeed we first give a construction of CSBP as scaling limits of continuous time G-W branching processes. To give a precise meaning to the convergence of trees, we will code G-W trees by a continuous exploration process as already defined by Dramé et al. in \cite{drame2016non}, and we will establish the convergence of these (rescaled) continuous process to the continuous height process defined in \cite{schneider2003forcing}, which is also the one defined in Chapter 1 of \cite{duquesne2002random}. 

In \cite{drame2016non} Dramé et al study the convergence of a general continuous time branching processes which describes a population where multiple births are allowed (in the case where the number of children born at a given birth event has a finite moment of order $2+\delta,$ for some $\delta>0$ arbitrarily small). In the present work, we aim to extend those results to G-W trees with possibly infinite variance of the numbers of children born at a given birth event. In this paper we use some recent results concerning the genealogical structure of CSBP that can be found in (\cite{le1998branching},\cite{duquesne2002random}, \cite{schneider2003forcing}). 

The organization of the present paper is as follows : In Section 2 we recall some basic definitions and notions concerning branching processes. In Section 3 and 4, we present our main results and the proofs. We shall assume that all random variables in the paper are defined on the same probability space $( \Omega,\mathcal{F},\mathbb{P})$. We shall use the following notations $\mathbb{Z}_+=\{0,1, 2,... \}$, $\mathbb{N}=\{1, 2,... \}$, $\mathbb{R}=(-\infty, \infty)$ and $\mathbb{R}_+=[0, \infty)$. For $x$ $\in$ $\mathbb{R}_+$, $[x]$ denotes the integer part of $x$. 

\section{Preliminaries}
\subsection{Continuous state branching process}
A CSBP is a $\mathbb{R}_+$-valued strong Markov process starting from the value $x$ at time $0$ whose probabilities $(\mathbb{P}_{x}, \ x\geqslant0)$ is such that for any $x$, $y\geqslant0$,  $\mathbb{P}_{x+y}$ is equal in law to the convolution of $\mathbb{P}_{x}$ and $\mathbb{P}_{y}$. More precisely, a CSBP $X^x=(X_t^x, \ t\geqslant0)$ (with initial condition $X_0^x=x$) is a Markov process taking values in $[0, \infty]$, where $0$ and $\infty$ are two absorbing states, and satisfying the branching property; that is to say, it's Laplace transform satisfies 
\begin{equation*}
   \mathbb{E} \left[ \exp (-\lambda X_t^x)\right] = \exp \left\{ -x u_t(\lambda)\right\}, \quad \mbox{for} \ \lambda\geqslant0, 
\end{equation*}
for some non negative function $u_t$. According to Silverstein \cite{silverstein1968new}, the function $u_t$ is the unique nonnegative solution of the integral equation
\begin{equation}\label{Utlamda}
 u_t(\lambda) = \lambda - \int_{0}^{t} \psi ( u_r(\lambda)) dr, 
\end{equation}
where $\psi$ is called the branching mechanism associated with $X^x$ and is defined by
\begin{equation}\label{PSIEXPRES}
 \psi(\lambda) = b\lambda +{c\lambda^2}+ \int_{0}^{\infty} (e^{-\lambda r}-1+\lambda r) \mu(dr), 
\end{equation}
where $b \in \mathbb{R}$, $c\geqslant0$ and $\mu$ is a $\sigma$-finite measure on $(0, \infty)$ which satisfies 
\begin{equation}\label{R2FINI}
\int_{0}^{\infty} (r\wedge r^2) \mu(dr) < \infty.
\end{equation}
We shall sometimes write $ \psi_{b,c,\mu}$ for the function $\psi$ attached to the triple $(b,c,\mu)$.

Let us recall that $b$ represents a drift term, $c$ is a diffusion coefficient and $\mu$ describes the jumps of the CSBP. The CSBP is then characterized by the triplet $(b, c, \mu)$ and can also be defined as the unique non negative strong solution of a stochastic differential equation. More precisely, from Fu and Li \cite{fu2010stochastic} (see also the results in Dawson-Li \cite{dawson2012stochastic}) we have 
\begin{equation}\label{XT}
X_t^x= X_0^x - b \int_{0}^{t} X_s^x ds + \sqrt{2c} \int_{0}^{t} \sqrt{X_s^x} dW_s + \int_{0}^{t}\int_{0}^{\infty}\int_{0}^{X_{s^-}^x}r \overline{M}(ds, dr, du),
\end{equation}
where $W$ is a standard Brownian motion, $M(ds, dr, du)$ is a Poisson random measure with intensity $ds\mu(dr)du$ independent of $W$, and $\overline{M}$ is the compensated measure of $M$. 
\begin{remark}\label{RmarPSI}
The assumption \eqref{R2FINI} is equivalent to $\psi$ being locally Lipschitz; see the proof of Proposition 1.45 of Li \cite{li2010measure}. This property plays an important role in what follows.
\end{remark}
The following result is Theorem 2.1.8 in Li \cite{li2012continuous}
\begin{proposition}
Suppose that $\psi$ is given by \eqref{PSIEXPRES}. Then there is a Feller transition semigroup $(Q_t)_{t\geqslant0}$ on $\mathbb{R}_{+}$ defined by 
\begin{equation}\label{QTRANSI}
 \int_{0}^{\infty} e^{-\lambda y} Q_t(x,dy) = e^{-x u_t(\lambda)}, \quad \lambda\geqslant0, \ x\geqslant0.
\end{equation}
A Markov process is called a CSBP with branching mechanism $\psi$ if it has transition semigroup $(Q_t)_{t\geqslant0}$ defined by \eqref{QTRANSI}.
\end{proposition} 
\subsection{The height process}
We shall also interpret below the function $\psi$ defined by \eqref{PSIEXPRES} as the Laplace exponent of a spectrally positive Lévy process $Y$. Lamperti \cite{lamperti1967continuous} observed that CSBPs are connected to Lévy processes with no negative jumps by a simple time-change. More precisely, define
\begin{equation*}
A_s^x=\int_{0}^{s}X_t^x dt, \quad \tau_s=\inf\{t>0, \ A_t^x >s\} \quad \mbox{and} \quad Y_s=X_{\tau_s}^x. 
\end{equation*}
Then $Y_s$ is a Lévy process of the form
\begin{equation}\label{YL\'{e}vy}
Y_s= - b s + \sqrt{2c} B_s + \int_{0}^{s}\int_{0}^{\infty}z \overline{\Pi}(dr, dz),
\end{equation}
where $B$ is a standard Brownian motion and $\overline{\Pi}(ds, dz)= \Pi(ds, dz)- ds\mu(dz)$, $\Pi$ being a Poisson random measure on $\mathbb{R}_{+}^{2}$ independent of $B$ with mean measure $ds\mu(dz)$. We refer the reader to \cite{lamperti1967continuous} and \cite{caballero2009continuous} for a proof of that result. In the sequel of this paper (in subsection 3.2), we will assume that $Y$ is a L\'{e}vy process with no negative jumps, whose Laplace exponent $\psi$ has the form \eqref{PSIEXPRES}, where $b\geqslant0$, $c>0$ and $\mu$ is a $\sigma$-finite measure on $(0,\infty)$ which satisfies \eqref{R2FINI}, and we exclude the case $\int_{(0,1)}r \mu(dr)< \infty$. We note that our standing assumption $c>0$ implies the Grey condition 
 \begin{equation*}\label{GreyCondi}
 \int_{1}^{\infty} \frac{d\lambda}{\psi(\lambda)}< \infty.
 \end{equation*}
This assumption ensures also that the corresponding height process $H$ is continuous, see \cite{duquesne2002random} (recall that if this condition does not hold, the paths of $H$ have a very wild behavior).  
To code the genealogy of the CSBP, Le Gall and Le Jan \cite{le1998branching} introduced the so-called height process, which is a functional of a L\'{e}vy process with Laplace exponent $\psi$; see also Duquesne and Le Gall \cite{duquesne2002random}. In this paper, we will use the new definition of the height process $H$ given by Li et all in \cite{schneider2003forcing}. Indeed, if the L\'{e}vy process $Y$ has the form \eqref{YL\'{e}vy}, then the associated height process is given by 
\begin{equation}\label{CHS1}
c H_s= Y_s - \inf_{0\leqslant r \leqslant s} Y_r - \int_{0}^{s}\int_{0}^{\infty}\left(z + \inf_{r\leqslant u \leqslant s}(Y_u-Y_r) \right)^+\Pi(dr, dz),
\end{equation}
and it has a continuous modification. Note that the height process $H_s$ is the one defined in Chapter 1 of  \cite{duquesne2002random}. We shall need the following result which is Lemma 1.3.2 in \cite{duquesne2002random}.
\begin{lemma}\label{HY}
For every $s>0$,
\begin{equation*}
 \lim_{\varepsilon\mapsto 0}\frac{1}{\varepsilon} \int_{0}^{s}\mathbf{1}_{ \{H_{r} \le \varepsilon \}}dr= -\inf_{0\leqslant r \leqslant s}Y_r.
\end{equation*}
\end{lemma}
Let us fix our notations concerning the local times. We define the local time accumulated by $H$ (the height process associated with the L\'{e}vy process $Y$) at level $t$ up to time $s$: 
\begin{equation}\label{loc}
  L_{s}(t) = \frac{2}{c}\lim_{\varepsilon\mapsto 0}\frac{1}{\varepsilon} \int_{0}^{s} \mathbf{1}_{\{t\le H_{r} < t+ \varepsilon \}}  (H_{r}) dr.
\end{equation}

Combining Lemma \ref{HY} with \eqref{loc} leads to 
\begin{remark}
The scaling of $H$ is such that $$\frac{1}{2} L_{s}(0)= -\frac{1}{c} \inf_{0\leqslant r \leqslant s}Y_r \quad \mbox{and} \quad \frac{1}{2} (L_{s}(H_{r}) - L_{r}(H_{r}))=-\frac{1}{c} \inf_{r\leqslant u \leqslant s}(Y_u-Y_r).$$
\end{remark}
\section{Scaling Limits of continuous time branching processes}
In this section, we obtain the CSBP as a scaling limit of continuous Galton-Watson branching processes. We will start with the general case and then we will treat a special case. Let $N\geqslant1$ be an integer which will eventually go to infinity. 
\subsection{The general case}
In this subsection, we obtain the general form of the branching mechanism of CSBP. We then provide a construction of these processes via an approximation by continuous-time Galton-Watson processes. To this end, let us define $L$ $\in$ $ \mathcal{C}([0, +\infty))$ by 
\begin{equation}\label{Ldeff}
L(u)= \int_{0}^{\infty} (e^{-u r}-1+ u r) \mu(dr), 
\end{equation}
where $\mu$ satisfies \eqref{R2FINI}. 
\begin{remark}
The family \eqref{Ldeff} contains the functions
\begin{equation*}
\ell (u) = const. u^{\gamma},\quad 1< \gamma < 2,
\end{equation*}
that correspond to $ \mu(dr)= const. r^{-(1+\gamma)}dr$ that we will develop in the next subsection but in a more general context.
\end{remark}
We set 
\begin{equation}\label{D1N}
d_{1,N}= \int_{0}^{\infty} r(1-e^{-Nr}) \mu(dr),
\end{equation}
and
\begin{equation}\label{genepi}
f_{1,N}(s)= s + N^{-1}d_{1,N}^{-1} L(N(1-s)), \quad \quad |s| \leqslant 1.
\end{equation}
It is easy to see that $s\rightarrow f_{1,N}(s)$ is an analytic function in $(-1,1)$ satisfying $f_{1,N}(1)=1$ and 
\begin{equation*}
\frac{d^n}{ds^n}f_{1,N}(0)\ge 0, \quad  n \ge 0.
\end{equation*}
Therefore $f_{1,N}$ is a probability generating function. Now, let $\xi_{1,N}$ be a random variable whose generating function is $f_{1,N}$. In what follows, we set $\mu_N=cN+ \beta$, $\lambda_N=cN+\alpha,$ and $d_{2,N}=\mu_N+\lambda_N$, where $\alpha, \beta, c\ge0.$ 

Let us define for $0\le s\le 1$,
\begin{equation}\label{f2N}
f_{2,N}(s)= \frac{1}{d_{2,N}}(\mu_N+ \lambda_N s^2).
\end{equation}
It is easy to check that $f_{2,N}$ is a probability generating function. Let $\xi_{N,2}$ be a random variable whose generating function is $f_{2,N}$. For the rest of this subsection we set, 
\begin{equation}\label{ConstN}
d_{N}= d_{1,N}+d_{2,N}, \quad \mbox{and} \quad b=\beta-\alpha.
\end{equation} 
Let $\epsilon_N$ be a random variable defined by
\begin{equation}\label{EpsilonN}
\mathbb{P}( \epsilon_N= k)= \frac{d_{k,N}}{d_{N}}, \quad  k \in \{1,2 \}.
\end{equation}
We assume that the three variables $\xi_{N,1}$, $\xi_{N,2}$ and $\epsilon_N$ are independent. Let $\eta_N$ be a random variable defined by
\begin{equation}\label{etadeff}
\eta_N= \xi_{N,1} \mathbf{1}_{\{\epsilon_N=1\}}+ \xi_{N,2} \mathbf{1}_{\{\epsilon_N=2\}}.
\end{equation}
Now we consider a continuous time $\mathbb{Z}_+$-valued branching process $Z^{N,x}=\{Z_t^{N,x}, \ t\geqslant0 \}$ which describes the population size at time $t$. In this population, each individual dies independently of the others at the constant rate $d_N$, and gives birth to $\eta_N$ new offspring individuals. In other words, from \eqref{etadeff} the generating function of the branching distribution is 
\begin{align}\label{hndeff}
h_N(s)&= \mathbb{E}(s^{\eta_N})=\mathbb{E}\left(s^{\xi_{N,1}}\mathbf{1}_{\{\epsilon_N=1\}}\right) + \mathbb{E}\left(s^{\xi_{N,2}}\mathbf{1}_{\{\epsilon_N=2\}}\right)=\frac{1}{d_{N}} \big( d_{1,N}f_{1,N}(s)+d_{2,N}f_{2,N}(s)\big). 
\end{align}
Such a process is a Bienaymé-Galton-Watson process in which to each individual is attached a random vector describing her lifetime and her number of offsprings. We assume that those random vectors are independent and identically distributed (i.i.d). The rate of reproduction is governed by a finite measure $\nu_N$ on $\mathbb{Z}_+$, satisfying $\nu_N(1)=0$ and  $\nu_N(k)=  d_{1,N} \mathbb{P}(\xi_{N,1}=k)+ d_{2,N} \mathbb{P}(\xi_{N,2}=k)$, for every $k\geqslant0$. More precisely, each individual lives for an exponential time with parameter $\nu_N(\mathbb{Z}_+)$, and is replaced by a random number of children according to the probability $\nu_N(k) (\nu_N(\mathbb{Z}_+))^{-1}$ for every $k\geqslant0$. Hence the dynamics of the continuous time Markov process $Z^{N,x}$ is entirely characterized by the measure $\nu_N$. We have the following proposition, which can be seen in Athreya-Ney \cite{athreya1972branching}; see also Pardoux \cite{pardoux1975equations}. 
\begin{proposition}\label{generaZN}
The generating function of the process $Z^{N,x}$ is given by 
\begin{equation*}
\mathbb{E}_{1}\left(s^{Z_t^{N,x}}\right)= w_t^N (s), \quad s \in [0,1], 
\end{equation*}
where
\begin{equation*}
 w_t^N (s) = s+ \int_{0}^{t}\Phi_N(w_r^N (s))dr 
\end{equation*}
and the function $\Phi_N$ id defined by 
\begin{align*}
\Phi_N(s)&= \sum_{k=0}^{\infty} (s^k-s) \nu_N(k)\\
&=\nu_N(\mathbb{Z}_+) (h_N(s)-s), \quad s \in [0,1],
\end{align*}
where $\nu_N(\mathbb{Z}_+)=d_{N}$ and $h_N$ is the generating function given by 
\begin{equation*}
h_N(s)= \sum_{k=0}^{\infty} q_k^N s^k, \quad q_k^N= \nu_N(k) (\nu_N(\mathbb{Z}_+))^{-1} \geqslant0 \quad and \quad \sum_{k=0}^{\infty} q_k^N=1.
\end{equation*}
(Recall that $h_N$ was also defined in \eqref{hndeff}).
\end{proposition}
We are interest in the scaling limit of the process $Z^{N,x}$ : We will start $Z^{N,x}$ with $Z_0^{N,x}= [Nx]$ for some fixed $x>0$, and study the behaviour of  $X_t^{N,x}= N^{-1}Z_t^{N,x}$. The continuous time process $\{X_t^{N,x}, \ t\geqslant0 \}$ is a Markov process with values in the  set $E_N= \{k / N, \ k\geqslant1 \}$. We denote by $(P_t^N, t\ge0)$ the transition probability of the process $X_t^{N,x}$. For $\lambda\geqslant0$
\begin{align*}
\int_{E_N} e^{-\lambda y}P_t^N(x, dy)&= \mathbb{E}\left( e^{-\lambda X_t^{N,x}}\Big| X_0^{N,x}\right)=  \mathbb{E}_{[Nx]}\left( e^{-\lambda ({Z_t^{N,x}}/{N})}\right) \nonumber\\
&= \exp \left([Nx] \log w_t^N \left(e^{-\lambda /N}\right) \right),
\end{align*}
where $w_t^N$ was given in Proposition \ref{generaZN}. This suggests to define 
\begin{equation}
u_{t}^{N}(\lambda)= N\left(1-w_t^N \left(e^{-\lambda /N}\right)\right).
\end{equation}
The function $u_{t}^{N}$ solves the equation 
\begin{equation}\label{UNtlamba}
 u_t^N(\lambda) + \int_{0}^{t} \psi^N ( u_r^N(\lambda)) dr= N\left(1- e^{-\lambda /N}\right), 
\end{equation}
where $\psi^N(u)= N\Phi_N(1-\frac{u}{N})$. However, from the definition of $\Phi_N$ in Proposition \ref{generaZN}, we have
\begin{equation}\label{PSINDEF}
\psi^N(u) = Nd_{N} \left( h_N\left(1-\frac{u}{N}\right)- \left(1-\frac{u}{N}\right)\right), \quad 0\le u\le N.
\end{equation}
Note that 
\begin{align*}
 X_{0}^{N,x}= {[Nx]}/{N}\longrightarrow x \quad  as \quad   N \rightarrow +\infty. 
\end{align*}
The following Lemma plays a key role in the asymptotic behavior of $X^{N,x}$
\begin{lemma}\label{PsiNversPSi}
The sequence $\psi^N(u)$  converges to $\psi(u)$ defined in \eqref{PSIEXPRES} as $N\longrightarrow \infty$.
\end{lemma}
\begin{proof}
Combinng  \eqref{hndeff} and \eqref{PSINDEF}, we have 
\begin{align*}
\psi^N(u)&= Nd_{1,N}\left( f_{1,N}\left(1-\frac{u}{N}\right)- \left(1-\frac{u}{N}\right)\right)+ Nd_{2,N}\left( f_{2,N}\left(1-\frac{u}{N}\right)- \left(1-\frac{u}{N}\right)\right)\\
&= \psi_{1,N}(u)+ \psi_{2,N}(u).
\end{align*}
From  \eqref{f2N} it is easy to check  that 
\begin{align*}
 \psi_{2,N}(u)&= bu+ cu^2+ \frac{\alpha}{N}u^2 
\end{align*}
Hence, it follows that the sequence $\psi_{2,N}(u)$ converges to $bu+cu^2$ as $N\longrightarrow \infty$. However, from \eqref{Ldeff} and \eqref{genepi} it is easy to see that $\psi_{1,N}(u)= L(u)$. The desired result follows readily by combining the above arguments.
\end{proof}
\begin{proposition}\label{Zenghu}
Let $(t, \lambda)\longrightarrow u_t(\lambda)$ be the unique locally bounded positive solution of \eqref{Utlamda}. Then we have for every $\lambda\geqslant0$, $ u_t^N(\lambda)\longrightarrow u_t(\lambda)$ uniformly on compact sets in $t$, as $N\longrightarrow \infty$. (recall that $u_t^N(\lambda)$ was given in \eqref{UNtlamba})
\end{proposition}
\begin{proof}
We take the difference between \eqref{Utlamda} and \eqref{UNtlamba}, and use Lemma \ref{PsiNversPSi} and the fact that $\psi$ is Lipschitz on $[0,\lambda e^{Tb^-} ]$ (see Remark \ref{RmarPSI}) to obtain that for $0\le t\le T$,
\begin{equation*} 
\big| u_t(\lambda)- u_t^N(\lambda)\big| \leqslant K_\lambda \int_{0}^{t} \big| u_s(\lambda)- u_s^N(\lambda)\big| ds + {k}_N(\lambda) + \frac{\alpha}{N}\int_{0}^{t} \left(u_s^N(\lambda) \right)^2 ds,
\end{equation*}
where ${k}_N(\lambda)= \lambda- N\left(1- e^{-\lambda /N}\right)\longrightarrow0$ as $N\longrightarrow \infty$, and $K_\lambda$ is the Lipschitz constant for $\psi$ on $[0,\lambda]$. We conclude from Gronwall's lemma that for every $\lambda\geqslant0$, 
\begin{equation*} 
\lim_{N\longrightarrow \infty} u_t^N(\lambda)= u_t(\lambda)
\end{equation*} 
uniformly on compact sets in $t$.
\end{proof}

Let $D([0,\infty), \mathbb{R}_+)$ denote the space of functions from $[0,\infty)$ into $\mathbb{R}_+$ which are right continuous and have left limits at any $t>0$ (as usual such a function is called càdlàg). We shall always equip the space $D([0,\infty),\mathbb{R}_+)$ with the Skorohod topology. The main limit theorem of this section is the following : 
\begin{theorem}\label{TH1}
 Let $\{X_t^x, \ t\geqslant0 \}$ be the càdlàg CSBP defined in \eqref{XT} with transition semigroup $(Q_t)_{t\geqslant0}$ defined by \eqref{QTRANSI}. Since $X_0^{N,x}$ converges to $X_0^x$, $\{X_t^{N,x}, \ t\geqslant0 \}$ converges to $\{X_t^x, \ t\geqslant0 \}$ in distribution on $D([0,\infty),\mathbb{R}_+)$.
\end{theorem}
\begin{proof}
The proof of the theorem follows by Proposition \ref{Zenghu} and  an argument similar to the proof of theorem 3.43 in Li \cite{li2010measure} (see, also Theorem 2.1.9 in Li  \cite{li2012continuous}).
\end{proof}
\subsection{A special case}
We now want to specify the above statement in particular case. In other words, in this case we give a special case of triplet $(b, c, \mu)$ characterizing the branching mechanism. To this end,  let  $f_\gamma$  and $f_2$ be two probability generating functions defined respectively by 
\begin{equation}\label{LESFUNC}
\left\{
    \begin{array}{ll}
   f_\gamma(s)= s+\gamma^{-1} (1-s)^{\gamma} 
   &\\\\ f_2(s)= \frac{1}{2}(1 + s^2), 
    
&
           \end{array}
           \right.
\end{equation}
$s$ $\in$ $[0,1]$, where $\gamma$ $\in$ $(1,2)$. Let $m$ be a probability measure on $(1,2)$.  We define 
 \begin{equation*}
\nu(d\gamma)=  m(d\gamma)+  \delta_2(d\gamma), \quad \gamma \ \in (1, 2],
\end{equation*}
where $\delta_2$ is the Dirac measure at $\gamma=2$.
For the rest of this section we set 
 \begin{equation}\label{RON}
\rho_N= \int_{(1,2]} C_\gamma \gamma N^{\gamma-1} \nu(d\gamma), \quad \mbox{with} \quad C_2= 1/2 \quad \mbox{and} \quad \sup_{1< \gamma < 2}C_\gamma<\infty.
\end{equation}
 Let  $\bar{\eta}_N$ be a random variable with in values in $\mathbb{Z}_+$, and $\bar{h}_N$ its probability generating function defined by 
 \begin{equation}\label{HNDEF}
 \bar{h}_N(s)= \rho_N^{-1} {\int_{(1,2]} C_\gamma \gamma N^{\gamma-1} f_\gamma(s) \nu(d\gamma)}, \quad s \ \in \ [0,1].
\end{equation}
\begin{remark}
The interest of this special case is that not only will we have an explicit measure $\mu$, which is the mean measure of a mixture of $\gamma$-stables processes. In other words, in the case where the measure $\nu$ is given by $\nu=\delta_\gamma$ (where $\delta_\gamma$ is the Dirac measure at $\gamma$), then we find the classical $\gamma$-stable case.
\end{remark}
We consider a continuous time $\mathbb{Z}_+$-valued branching process $\bar{Z}^{N,x}=\{\bar{Z}_t^{N,x}, \ t\geqslant0 \}$ which describes the population size at time $t$. In this population, each individual dies independently of the others at constant rate $\rho_N$, and gives birth to $\bar{\eta}_N$ new offspring individuals.We now define the rescaled continuous time process
\begin{equation*}
\bar{X}_{t}^{N,x}:=N^{-1}\bar{Z}_{t}^{N,x}.
\end{equation*}
In particular, we have that 
\begin{align*}
 \bar{X}_{0}^{N,x}= {[Nx]}/{N}\longrightarrow x \quad  as \quad   N \rightarrow +\infty. 
\end{align*}
Following the same approach as general case, the approximate branching mechanism defined in \eqref{PSINDEF} is obtained by an easy adaptation. In other words, in this case, the equation \eqref{PSINDEF} takes the following form
\begin{equation}\label{PSINBARDEF}
\bar{\psi}^N(u) = N \rho_N \left( \bar{h}_N\left(1-\frac{u}{N}\right)- \left(1-\frac{u}{N}\right)\right), \quad 0\le u\le N.
\end{equation}
We now prove
\begin{lemma}\label{PsiNversPSiSpecial}
The sequence $\bar{\psi}^N(u)$ converges to 
\begin{equation}\label{PSIEXPres}
\bar{\psi}(u)= u^2+ \int_{0}^{\infty} (e^{-r u}-1+r u)\mu(dr)
\end{equation}
 as $N\longrightarrow \infty$,  where $$\mu(dr)= \left( \int_{(1,2)} C_\gamma  \frac{\gamma(\gamma-1)}{\Gamma(2-\gamma)} \frac{m(d\gamma)}{r^{\gamma+1}}\right)dr.$$
\end{lemma}
\begin{proof}
Combining \eqref{LESFUNC}, \eqref{RON}, \eqref{HNDEF} and \eqref{PSINBARDEF}, we have 
\begin{equation*}
\bar{\psi}^N(u)= \bar{\psi}_1^N(u)+ \bar{\psi}_2^N(u)
\end{equation*}
with
\begin{equation}\label{PSI0N}
\bar{\psi}_1^N(u)= N  \int_{(1,2)} C_\gamma \gamma N^{\gamma-1} \left[  f_\gamma \left(1-\frac{u}{N}\right)- \left(1-\frac{u}{N}\right)\right]m(d\gamma)
\end{equation}
and  
\begin{equation*}
\bar{\psi}_2^N(u)= N^2 \left(  f_2\left(1-\frac{u}{N}\right)- \left(1-\frac{u}{N}\right)\right).
\end{equation*}
In the same way as done in the proof of Lemma \ref{PsiNversPSi}, we have that the sequence $\bar{\psi}_2^N(u)$ converges to $u^2$ as $N\longrightarrow \infty$.
However, from \eqref{LESFUNC} and \eqref{PSI0N}, it is easily to see that $\bar{\psi}_1^N(u)= \int_{(1,2)} C_\gamma u^{\gamma} m(d\gamma)$. Now, noting that  
$$u^\gamma= \frac{\gamma(\gamma-1)}{\Gamma(2-\gamma)}  \int_{0}^{\infty} (e^{-r u}-1+r u)\frac{dr}{r^{\gamma+1}}, \quad u\geqslant0 \quad 1<\gamma<2,$$
we deduce from Fubini's Theorem that 
$$\bar{\psi}_1^N(u)=  \int_{0}^{\infty} (e^{-r u}-1+r u)\mu(dr), \quad \mbox{where} \quad \mu(dr)= \left( \int_{(1,2)} C_\gamma  \frac{\gamma(\gamma-1)}{\Gamma(2-\gamma)} \frac{m(d\gamma)}{r^{\gamma+1}}\right)dr.$$
The desired result follows readily by combining the above arguments.
\end{proof}

For the convenience of statement of the result, we assume that $m$ satisfies the following  condition
\begin{equation}\label{Condi}
 \int_{(1,2)} C_\gamma  \frac{\gamma(\gamma-1)}{(2-\gamma)\Gamma(2-\gamma)} m(d\gamma)< \infty.
\end{equation}
Condition \eqref{Condi} combined with Remark \ref{RmarPSI} leads to
\begin{corollary}\label{PSILIP}
The function $\bar{\psi}$ $\in$ $C([0, +\infty))$ with the representation \eqref{PSIEXPres} is locally Lipschitz.
\end{corollary}
The rest is entirely similar to the general case. Therefore, we obtain a similar convergence result.
\begin{theorem}
 Let $\{\bar{X}_t^x, \ t\geqslant0 \}$ be a càdlàg CSBP defined as \eqref{XT} with transition semigroup $(Q_t)_{t\geqslant0}$ defined by \eqref{QTRANSI}. Since $\bar{X}_0^{N,x}$ converges to $\bar{X}_0^x$ in distribution,  then $\{\bar{X}_t^{N,x}, \ t\geqslant0 \}$ converges to $\{\bar{X}_t^x, \ t\geqslant0 \}$, where the triplet $(b,c,\mu)$ is given by $$b=0, \ c=1 \quad and \quad \mu(dr)= \left( \int_{(1,2)} C_\gamma  \frac{\gamma(\gamma-1)}{\Gamma(2-\gamma)} \frac{m(d\gamma)}{r^{\gamma+1}}\right)dr.$$ 
The convergence holds in the sense of weak convergence on $D([0,\infty),\mathbb{R}_+)$.
\end{theorem}
\section{Convergence of the Exploration process}
In this section, we show that the rescaled exploration process of the corresponding Galton-Watson family tree, converges in a functional sense, to the continuous height process associated with the CSBP. In this section, we assume that 
$$c>0, \quad b\ge0 \quad (i.e \ \beta >\alpha),$$
and we renforce \eqref{R2FINI}, and assume that for some $1<p<2$,
\begin{equation*}
{(\bf H)}  :  \quad  \int_{0}^{\infty} (r \vee r^{p})  \mu(dr)< \infty.
\end{equation*} 
Let us rewrite \eqref{YL\'{e}vy} in the following form 
\begin{equation*}
Y_s= Y_s^c+ \int_{0}^{s}\int_{0}^{\infty}z {\Pi}(dr, dz), \quad \mbox{with} \quad Y_s^c= - (b +\int_{0}^{\infty}z \ \mu(dz))s + \sqrt{2c} B_s.
\end{equation*}
Consequently, we can rewrite \eqref{CHS1} in the form 
\begin{equation}\label{CHS2}
c H_s= Y_s^c - \inf_{0\leqslant r \leqslant s} Y_r + \int_{0}^{s}\int_{0}^{\infty}  z\wedge \left(- \inf_{r\leqslant u \leqslant s}(Y_u-Y_r) \right)\Pi(dr, dz)
\end{equation}
\begin{remark}
Note that the last term on the right end side of \eqref{CHS2} is an continuous and increasing process. And we notice also that the second writing of $H$ is possible thanks to the assumption ${(\bf H)}$.
\end{remark}
Let us note that according to an inequality due to Li et all in \cite{schneider2003forcing}, we have
$$\E  \int_{0}^{s}\int_{0}^{\infty}\left(z + \inf_{r\leqslant u \leqslant s}(Y_u-Y_r) \right)^+\Pi(dr, dz) \le C(s) \int_{0}^{\infty}  (z\wedge z^2) \mu(dz).$$
So the first writing of $H$ has a meaning without the supplementary assumption ${(\bf H)}$. But we were not able to establish the convergence of the exploration process without the assumption ${(\bf H)}$. 

The measure $\mu$ will appear many times in this section. It will always refer to a measure on $\mathbb{R}_+$ satisfying ${(\bf H)}$.
\begin{figure}[H]
\centering
\includegraphics[width=4.5in]{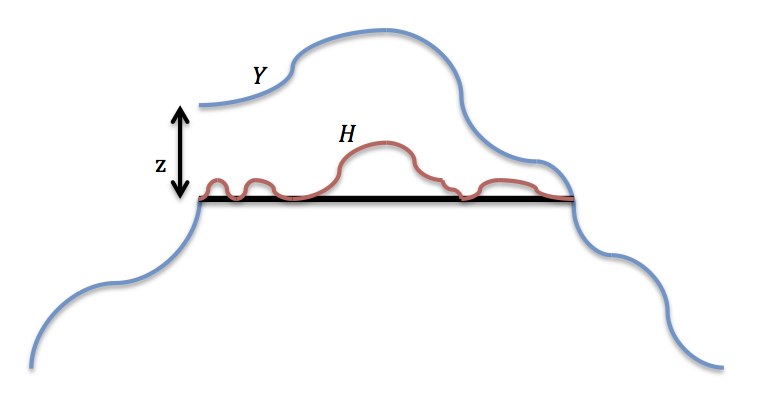}
\caption{Trajectories of $Y$ and $H$.}
\end{figure} 
 Let us state some intermediate results which will be useful in the sequel.
\subsection{Preliminary results}
We notice that one of the aims of this subsection is construct the random measure $\widetilde{\pi}_{1,N}$, which will be specified below in \eqref{GamaN1N2}. It is an complicated construction but essential for the rest.

Let us define 
$L^{1,-}$ and $L^{1,+}$ $\in$ $ \mathcal{C}([0, +\infty))$ by 
\begin{equation}\label{L+L-}
L^{1,-}(u)= \int_{0}^{1} (e^{-u r}-1+ u r) \mu(dr) \quad \mbox{and} \quad L^{1,+}(u)= \int_{1}^{\infty} (e^{-u r}-1+ u r) \mu(dr), 
\end{equation}
where $\mu$ satisfies \eqref{R2FINI}. In what follows, we set 
\begin{equation*}
\alpha_{-,N}= \int_{0}^{1} r(1-e^{-Nr}) \mu(dr), \quad \alpha_{+,N}= \int_{1}^{\infty} r(1-e^{-Nr}) \mu(dr), 
\end{equation*}
\begin{equation*} 
 h_{-,N}(s)= s + \frac{1}{ N\alpha_{-,N}} L^{1,-}(N(1-s)) \quad \mbox{and} \quad h_{+,N}(s)= s + \frac{1}{N\alpha_{+,N}} L^{1,+}(N(1-s)), \quad |s| \leqslant 1.
 \end{equation*}
Note that $d_{1,N}=\alpha_{-,N}+\alpha_{+,N}$, where $d_{1,N}$ was defined in $\eqref{D1N}$. From an adaptation of the argument used after equation \eqref{genepi}, we deduce that $h_{-,N}$ and $h_{+,N}$ are probability generating functions. We define $q_k^{-,N}$ and $q_k^{+,N}$ by
\begin{equation*}
q_k^{-,N}= h_{-,N}^{(k)}(0)/k! \quad \mbox{and} \quad q_k^{+,N}= h_{+,N}^{(k)}(0)/k!,  \quad k=0,1,2,..., 
\end{equation*}
where $h_{-,N}^{(k)}$ and $h_{+,N}^{(k)}$ denote the $k$-th derivative of $h_{-,N}$ and $h_{+,N}$ respectively. Hence it is well known that  $h_{-,N}$ and $h_{+,N}$ can be written as 
\begin{equation*}
h_{-,N}(s)= \sum_{k\ge0} q_k^{-,N} s^k  \quad \mbox{and} \quad h_{-,N}(s)= \sum_{k\ge0} q_k^{+,N} s^k,  \quad |s| \leqslant 1.
\end{equation*}
However, it is easy to check that 
\begin{equation}\label{q0-N}
q_0^{-,N}= \frac{1}{N\alpha_{-,N}} L^{1,-}(N), \quad q_1^{-,N}=0 \quad \mbox{and} \quad q_k^{-,N}= \frac{1}{k! N\alpha_{-,N}}\int_{0}^{1}(Nr)^k e^{-Nr} \mu(dr), \ \mbox{for} \ k\geqslant 2
\end{equation}
and
\begin{equation}\label{q0+N}
q_0^{+,N}= \frac{1}{N\alpha_{+,N}} L^{1,+}(N), \quad q_1^{+,N}=0 \quad \mbox{and} \quad q_k^{+,N}= \frac{1}{k! N\alpha_{+,N}}\int_{1}^{\infty}(Nr)^k e^{-Nr} \mu(dr), \ \mbox{for} \ k\geqslant 2.
\end{equation}
Now, let $\xi_{-,N}$ and $\xi_{+,N}$ be two random variables with the generating functions $h_{-,N}$ and $h_{+,N}$ respectively.
Let $\tilde{\epsilon}_N$ be a random variable defined by
\begin{equation*}
\tilde{\epsilon}_N=
\left\{
    \begin{array}{ll}
   1 \quad \mbox{with probability} \quad \alpha_{+,N}/d_{1,N}, &\\\\ 0   \quad \mbox{with probability} \quad  \alpha_{-,N}/d_{1,N}.
    
&
           \end{array}
           \right.
\end{equation*}
 We assume that the three variables $\xi_{-,N}$, $\xi_{+,N}$ and $\tilde{\epsilon}_N$ are independent. Let $\Gamma_{1,N}$ be a random variable defined by
\begin{equation}\label{GamaN}
\Gamma_{1,N}= \xi_{-,N} \mathbf{1}_{\{\tilde{\epsilon}_N=0\}}+ \xi_{+,N} \mathbf{1}_{\{\tilde{\epsilon}_N=1\}}.
\end{equation}
We denote by $g_{1,N}$ the probability generating function of $\Gamma_{1,N}$.  We deduce from \eqref{GamaN} that 
 \begin{equation*}
g_{1,N}(s)= \sum_{k\ge0} q_k^{1,N} s^k, \quad \quad |s| \leqslant 1,
\end{equation*}
where 
 \begin{equation}\label{qN1}
 q_k^{1,N}= \frac{1}{d_{1,N}} \left[q_k^{-,N}  \alpha_{-,N} + q_k^{+,N} \alpha_{+,N}   \right].
 \end{equation}
 Let us rewrite $g_{1,N}$ in the form
\begin{align}\label{gN1}
g_{1,N}(s)&= \frac{1}{d_{1,N}} \left[ h_{-,N}(s) \alpha_{-,N} +  h_{+,N}(s) \alpha_{+,N}   \right] \nonumber \\
&=s + N^{-1}d_{1,N}^{-1} L(N(1-s)),
\end{align}
where $L$  was defined in \eqref{Ldeff}. We notice that $g_{1,N}=f_{1,N}$, (recall that $f_{1,N}$ was defined in \eqref{genepi}). It is plain as previously that 
\begin{equation}\label{Qn10}
q_0^{1,N}= \frac{1}{Nd_{1,N}} L(N), \quad q_1^{1,N}=0 \quad \mbox{and} \quad q_k^{1,N}= \frac{1}{k! N d_{1,N}}\int_{0}^{\infty}(Nr)^k e^{-Nr} \mu(dr), \ \mbox{for} \ k\geqslant 2.
\end{equation}
Let $\Gamma_{2,N}$ be a random variable whose generating function $f_{2,N},$ which was defined in \eqref{f2N}. Let us define 
 \begin{equation*}
q_0^{2,N}= \frac{\mu_N}{d_{2,N}} , \quad q_2^{2,N}=\frac{\mu_N}{d_{2,N}} \quad \mbox{and} \quad q_k^{2,N}= 0  \  \ \mbox{for all} \ k\notin \{0,2\}.
\end{equation*}
 Hence, it is easy to see that $f_{2,N}$ can be written as 
\begin{equation*}
f_{2,N}(s)= \sum_{k\ge0} q_k^{2,N} s^k.
\end{equation*}
We assume that the three variables $\Gamma_{1,N} $, $ \Gamma_{2,N}$ and $\epsilon_N$ are independent, recall that $\epsilon_N$ was defined in \eqref{EpsilonN}. Let $\tilde{\eta}_N$ be a random variable defined by
\begin{equation*}
\tilde{\eta}_N= \Gamma_{1,N} \mathbf{1}_{\{\epsilon_N=1\}}+ \Gamma_{2,N} \mathbf{1}_{\{\epsilon_N=2\}}.
\end{equation*}
We denote by $f_{N}$ the probability generating function of $\tilde{\eta}_N$. Hence, it is easy to see that  
\begin{equation*}\label{fN(s)}
f_N(s)= \sum_{k\ge0} q_k^{N} s^k,
\end{equation*}
where 
\begin{equation*}\label{qN}
 q_k^{N}= \frac{1}{d_{N}} \left[q_k^{1,N}  d_{1,N} + q_k^{2,N} d_{2,N}   \right].
\end{equation*}
In other words, $ f_N $ can be written in the form
\begin{equation*}\label{fN2(s)}
f_N(s)= \frac{1}{d_{N}} \left[d_{1,N} g_{1,N}(s)   +  d_{2,N}  f_{2,N}(s) \right].
\end{equation*}
We notice that $f_{N}=h_{N}$, (recall that $h_{N}$ was defined in  \eqref{hndeff}). 
Now, let $\Theta^N$ be a random variable with with probability distribution  $q^{1,N}$ and let $\Lambda^{1,N}$ be a random variable with value in $\mathbb{N}$ defined by
\begin{equation*}
\Lambda^{1,N}= \mathcal{U}^N -1, \quad \mbox{where} \quad \mathcal{U}^N \ \sim \Theta^N \big|_{\Theta^N >0} . 
\end{equation*}
It is easy to check that 
\begin{align}\label{Ptheta}
\mathbb{P}(\Lambda^{1,N}=k)&= \frac{ q_{k+1}^{1,N}}{1- q_0^{1,N}}. 
\end{align}
In what follows, we set  $\mathbb{P}(\Lambda^{1,N}=k)=p_{k}^{1,N}.$ However, it is also easy to check that, for all $k$ $\in$ $\mathbb{N}$, 
\begin{equation*}
  0\le p_{k}^{1,N} \le 1 \quad  \mbox{and} \quad \sum_{k\ge1}p_{k}^{1,N}=1.
\end{equation*} 
Therefore, $p^{1,N}$ is a probability distribution on $\mathbb{N}.$ However, from \eqref{qN1} and \eqref{Ptheta}, we deduce that  
\begin{align}\label{pN1}
p_{k}^{1,N}&= \frac{1}{d_{1,N}(1- q_0^{1,N})} \left[q_{k+1}^{-,N}  \alpha_{-,N} + q_{k+1}^{+,N} \alpha_{+,N}   \right] \nonumber \\
&= \frac{1}{d_{1,N}(1- q_0^{1,N})} \left[p_{k}^{-,N}  \alpha_{-,N}(1- q_0^{-,N}) + p_{k}^{+,N} \alpha_{+,N}(1- q_0^{+,N})   \right]
\end{align}
with  
\begin{equation}\label{pN-pN+}
 p_{k}^{-,N}= \frac{ q_{k+1}^{-,N}}{1- q_0^{-,N}} \quad \mbox{and}  \quad p_{k}^{+,N}= \frac{ q_{k+1}^{+,N}}{1- q_0^{+,N}}.
\end{equation}
For the same arguments as previously $p^{-,N}$ and $p^{+,N}$ are probabilities distributions on ${\mathbb{N}}.$ Let $\Gamma_N^-$ and $\Gamma_N^+$ be two random variables with probabilities distributions $p^{-,N}$ and $p^{+,N}$ respectively. Let $\hat{\epsilon}_N$ be a random variable defined by
\begin{equation*}
\hat{\epsilon}_N=
\left\{
    \begin{array}{ll}
   1 \quad \mbox{with probability} \quad \alpha_{+,N}(1- q_0^{+,N})/d_{1,N}(1- q_0^{1,N}), &\\\\ 0   \quad \mbox{with probability} \quad  \alpha_{-,N}(1- q_0^{-,N})/d_{1,N}(1- q_0^{1,N}).
    
&
           \end{array}
           \right.
\end{equation*}
 We assume that the three variables $\Gamma_N^-$, $\Gamma_N^+$ and $\hat{\epsilon}_N$ are independent. From \eqref{pN1} it is easy to see that $\Lambda^{1,N}$ can be written as
\begin{equation}\label{Lamda1N}
\Lambda^{1,N}= \Gamma_N^+ \mathbf{1}_{\{\hat{\epsilon}_N=1\}}+ \Gamma_N^- \mathbf{1}_{\{\hat{\epsilon}_N=0\}}.
\end{equation}
We define $m_{1,N}= \E(\Lambda^{1,N})$ , the expectations of $\Lambda^{1,N}$. From \eqref{Ptheta}, we have the
\begin{remark}\label{RM000}
\begin{equation*}
 m_{1,N}= \frac{q_0^{1,N}}{1- q_0^{1,N}}.
\end{equation*}
\end{remark}
 Let  $\pi_{1,N}$  be a probability measure on $\mathbb{N}$ defined by  $\pi_{1,N}(\{k\})= p_k^{1,N}$. For the rest of this subsection, we set  
\begin{equation}\label{GamaN1N2}
\gamma_{1,N}= \frac{d_{1,N}q_0^{1,N}}{m_{1,N}}=d_{1,N}(1-q_0^{1,N})\quad \mbox{and}  \quad\widetilde{\pi}_{1,N}\left(\left\{ \frac{k}{N}\right\}\right)=N\gamma_{1,N} \pi_{1,N}(\{k\}).
\end{equation} 
In what follows, we will use the bounds 
\begin{equation}\label{Born}
0\le 1-e^{-\lambda} \le 1\wedge \lambda, \quad 0\le e^{-\lambda} -1+ \lambda \le \lambda \wedge \lambda^2, \quad \mbox{for all} \  \lambda \ge 0.  
\end{equation}
 We shall need below the
\begin{lemma}\label{d1Nq01N} 
For all $N\ge1$ $$\int_{0}^{\infty}z\widetilde{\pi}_{1,N}(dz)\le \int_{0}^{\infty} z \mu(dz) < \infty.$$
\end{lemma}
\begin{proof}
We have $$\int_{0}^{\infty}z\widetilde{\pi}_{1,N}(dz)=  N\gamma_{1,N} \sum_{k\ge1} \left(\frac{k}{N} \right) p_k^{1,N}= \gamma_{1,N} m_{1,N}=d_{1,N}q_0^{1,N} ,$$ where we have used \eqref{GamaN1N2}. However from \eqref{Qn10}, we deduce that 
\begin{align*}
d_{1,N}q_0^{1,N}=  \frac{L(N)}{N}.
\end{align*}
However, we deduce from \eqref{Born} that 
\begin{equation*}
\frac{L(N)}{N}= \frac{1}{N} \int_{0}^{\infty} (e^{-Nz}-1+ Nz) \mu(dz) \le \int_{0}^{\infty} z \mu(dz).
\end{equation*}
The desired result follows. 
\end{proof}

The end of the present section will be devoted to the proof of  
\begin{proposition}\label{PiTildversPiAA}
For every continuous function $\varphi$ from $\mathbb{R}_+$ into $\mathbb{R}$ such that $\left|\varphi(z) \right| \le C z$ for some constant $C,$ we have $$\widetilde{\pi}_{1,N}(\varphi) \longrightarrow  \mu (\varphi), {~}as{~}N\longrightarrow \infty. $$ 
\end{proposition}

We will need the following technical results on Galton-Watson trees. To this end, let $\mathcal{X}^{1,N}= (\mathcal{X}_k^{1,N}, \ k=0,1,2,...)$ be a sequence of Galton-Watson branching process with offspring distribution $(q_k^{1,N})_{k\geqslant 0}$, with offspring generating function $g_{1,N}.$ Recall that $q^{1,N}$ and $g_{1,N}$ were defined in \eqref{qN1} and \eqref{gN1} respectively. We have the following 
 \begin{proposition}\label{Conv1}
 (Theorem 3.43 in Li \cite{li2012continuous})\\
For any $t\geqslant 0$, 
\begin{equation*}
 \left(N^{-1}\mathcal{X}_{[d_{1,N}t]}^{1,N}, \ t\ge0\right)\Longrightarrow\left(\mathcal{X}_t, \ t\ge0\right) {~} in{~} \mathcal{D}([0,\infty)), {~}as{~}N\longrightarrow \infty, 
 \end{equation*}
 where $\mathcal{X}$ is a $\psi_{0,0,\mu}$-CSBP.
 \end{proposition}

We define another probability measure $\lambda_{1,N}$ on $\{-1,0,1,2,... \}$ by setting $\lambda_{1,N}(k)=q_{k+1}^{1,N}$ for every $k\ge -1$. We denote by $\mathcal{W}^{1,N}=(\mathcal{W}_k^{1,N}, \ k=1,2,...)$ a discrete-time random walk on $\mathbb{Z}$ with jump distribution $\lambda_{1,N}$ and started at $0$. We get the following result, which plays an important role in our approach. 
 \begin{proposition}
 (Theorem 2.1.1 in Duquesne and Le Gall \cite{duquesne2002random})\\
Proposition \ref{Conv1} implies that for any $s\geqslant 0,$ 
\begin{equation}\label{Conv2p}
 \left(N^{-1} \mathcal{W}_{[Nd_{1,N}s]}^{1,N} , \ s\ge0\right)\Longrightarrow\left( \mathcal{W}_s, \ s\ge0\right) {~} in{~} \mathcal{D}([0,\infty)), {~}as{~}N\longrightarrow \infty, 
 \end{equation}
 where $ \mathcal{W}$ is a $\psi_{0,0,\mu}$-L\'{e}vy process.
 \end{proposition} 

Let $f_0$ be a truncation function, that is a bounded continuous function from $\mathbb{R}$ into $\mathbb{R}$ such that $f_0(z) = z$ for every $z$ belonging to a neighborhood of $0.$ By standard results on the convergence of rescaled random walks (see e.g. Theorem II.3.2 in \cite{jacod1985theoremes}, see also the proof of Theorem 2.2.1 in  \cite{duquesne2002random}), the convergence \eqref{Conv2p} implies the two following condition satisfied:
\begin{align*}
{(\bf C)} & :  \quad \lim_{N\rightarrow\infty} Nd_{1,N} \sum_{k\ge1} f_0(\frac{k}{N})^2 \lambda_{1,N}(k)= \int_{0}^{\infty} f_0(z)^2\mu(dz).
\end{align*}
However, from \eqref{Ptheta} and \eqref{GamaN1N2}, we have that
  \begin{align*}
\widetilde{\pi}_{1,N}(f)&=\int_{0}^{\infty} f(z) \widetilde{\pi}_{1,N}(dz)\nonumber\\&= N\gamma_{1,N} \sum_{k\ge1}   p_k^{1,N} f({\frac{k}{N}}) \nonumber\\
&= Nd_{1,N} \sum_{k\ge1}   q_{k+1}^{1,N} f({\frac{k}{N}}) \nonumber \\
&=Nd_{1,N} \sum_{k\ge1} \lambda_{1,N}(k) f({\frac{k}{N}}). 
\end{align*}
Let us define $\bar{f} \ : \mathbb{R}_+\rightarrow\mathbb{R}_+$ by $\bar{f}(z)=z^2 \wedge 1$ and we set for all $h$  $\in$  $\mathcal{C}_b(\mathbb{R}^+)$
\begin{equation}\label{TO}
 \bar{\pi}_{N}(h)=\widetilde{\pi}_{1,N}(h \times \bar{f}).
\end{equation}
Note that  $(\bar{\pi}_{N},\ N\ge1)$ is also a finite measure on $\mathbb{R}_+.$ Furthermore, a simple consequence of ${(\bf C)}$ implies that for any  $h$  $\in$  $\mathcal{C}_b(\mathbb{R}^+)$ such that $h(z)=1$
on a neighborhood of $0,$ 
\begin{equation}\label{TEE}
 \bar{\pi}_{N}(h) \longrightarrow  \bar{\pi} (h), {~}as{~}N\longrightarrow \infty, 
\end{equation}
where $\bar{\pi} (h)=\int_{0}^{\infty} h(z) \bar{\pi}(dz)$ and  $\bar{\pi}(dz)= \bar{f}(z) \mu(dz).$

We shall need below the
\begin{lemma}\label{lemmaA}
For any $\epsilon>0,$ there exists $M>0$ such that
\begin{equation*}
\limsup_{N\rightarrow\infty} \ \bar{\pi}_N(K^c) \le \frac{\epsilon}{2}, 
\end{equation*} 
where $K=(0,M].$
\end{lemma}
\begin{proof}
Let us define  
 \begin{equation*}
g_1(z)=
\left\{
    \begin{array}{ll}
   z^2 \wedge 1, \quad \mbox{if} \quad z\le M-1, &\\\\  M-z, \quad  \mbox{if} \quad M-1< z < M
   &\\\\  0, \quad  \mbox{if} \quad z \ge M
    
&
           \end{array}
           \right.
     \mbox{and} \quad h(z)=\frac{g_1(z)}{\bar{f}(z)}=
\left\{
    \begin{array}{ll}
1, \quad \mbox{if} \quad z\le M-1, &\\\\  M-z, \quad  \mbox{if} \quad M-1< z < M
   &\\\\  0, \quad  \mbox{if} \quad  z \ge M.
    
&
           \end{array}
           \right.
\end{equation*}
Now, since $ \mathbf{1}_{K^c }\le 1-h \le \mathbf{1}_{[0,M-1]^c}=\mathbf{1}_{ (M-1,\infty)},$ we then obtain 
\begin{equation}\label{Pin1bar}
\bar{\pi}_N(K^c) \le  \bar{\pi}_N(1-h)=  \bar{\pi}_N(1)- \bar{\pi}_N(h).
\end{equation}
However, from \eqref{TEE}, it follows that
\begin{equation}\label{Pin1bar12}
 \bar{\pi}_N(1)- \bar{\pi}_N(h)\longrightarrow  \bar{\pi} (1)-\bar{\pi} (h), {~}as{~}N\longrightarrow \infty.
\end{equation}
It is plain that 
\begin{align*} 
\bar{\pi} (1)-\bar{\pi} (h)= \bar{\pi} (1-h) &\le \bar{\pi} ((M-1,\infty))\\
&\le \mu  ((M-1,\infty))\\
&\le  \frac{1}{M-1} \int_{(M-1,\infty)} z \ \mu(dz)\\
&\le   \frac{1}{M-1} \int_{0}^{\infty} z \ \mu(dz)\\
& \le \frac{C}{M-1}.
\end{align*} 
Combining this with \eqref{Pin1bar} and \eqref{Pin1bar12}, we deduce that 
\begin{equation*}
\limsup_{N\rightarrow\infty} \ \bar{\pi}_N(K^c) \le \frac{C}{M-1}. 
\end{equation*} 
It follows that for any $\epsilon>0,$ there exists $M>0$ such that
\begin{equation*}
\limsup_{N\rightarrow\infty} \ \bar{\pi}_N(K^c) \le \frac{\epsilon}{2}. 
\end{equation*} 
\end{proof}

Consequently, there exists $N_0$ such that for any $N\ge N_0,$  $\bar{\pi}_N(K^c)\le \epsilon. $ We then obtain the
\begin{corollary}
The sequence $(\bar{\pi}_{N},\ N\ge1)$ is tight.
\end{corollary}
Consequently there exists a subsequence $\bar{\pi}_{N}$ (which we denote as the whole sequence, as an abuse notation) which converges weakly. We prove the 
\begin{lemma}
As $N\longrightarrow \infty,$ $\bar{\pi}_{N}\left([0,z]\right)\longrightarrow  \bar{\pi} \left([0,z]\right),$ for every continuity point $z$ of $F_{\bar{\pi}}(z)= \bar{\pi} \left([0,z]\right)$.
\end{lemma}
\begin{proof}
Let us define  
 \begin{equation*}
\underline{g}_\epsilon(z)=
\left\{
    \begin{array}{ll}
    1, \quad \mbox{if} \quad 0\le y \le z-\epsilon, &\\\\  \frac{1}{\epsilon}(z-y), \quad  \mbox{if} \quad z-\epsilon < y < z
   &\\\\  0, \quad  \mbox{if} \quad y \ge z
    
&
           \end{array}
           \right.
     \mbox{and} \quad \overline{g}_\epsilon(z)=
\left\{
    \begin{array}{ll}
1, \quad \mbox{if} \quad 0\le y\le z, &\\\\  \frac{1}{\epsilon}(z+\epsilon-y), \quad  \mbox{if} \quad z < y < z +\epsilon
   &\\\\  0, \quad  \mbox{if} \quad  y \ge z+\epsilon.
    
&
           \end{array}
           \right.
\end{equation*}
From \eqref{TEE}, we deduce easily that $$\bar{\pi}_{N}\left(\underline{g}_\epsilon \right) \longrightarrow \bar{\pi}\left(\underline{g}_\epsilon \right) \quad \mbox{and} \quad \bar{\pi}_{N}\left(\overline{g}_\epsilon \right) \longrightarrow \bar{\pi}\left(\overline{g}_\epsilon\right), {~}as{~}N\longrightarrow \infty.$$ However, we have 
\begin{equation*}
 \bar{\pi}\left(\underline{g}_\epsilon \right) = \lim_{N\rightarrow\infty} \bar{\pi}_{N}\left(\underline{g}_\epsilon \right) \le \liminf_{N\rightarrow\infty} \bar{\pi}_{N}\left([0,z]\right) \le  \limsup_{N\rightarrow\infty} \bar{\pi}_{N}\left([0,z]\right) \le \lim_{N\rightarrow\infty} \bar{\pi}_{N}\left(\overline{g}_\epsilon \right)= \bar{\pi}\left(\overline{g}_\epsilon\right).
\end{equation*}
Therefore, the result follows by letting $\epsilon$ tend to $0,$ since provided $z$ is a continuity point of  $F_{\bar{\pi}}$, $$ \bar{\pi} \left([0,z]\right)= \lim_{\epsilon \rightarrow0}  \bar{\pi}\left(\underline{g}_\epsilon \right)=  \lim_{\epsilon \rightarrow0}  \bar{\pi}\left(\overline{g}_\epsilon\right).$$
\end{proof}

Since all converging subsequences have the same limit, the whole sequence converges :  
\begin{corollary}\label{CorpiNbar}
For any  $h$  $\in$  $\mathcal{C}_b(\mathbb{R}^+)$, $\bar{\pi}_{N}(h) \longrightarrow  \bar{\pi} (h), {~}as{~}N\longrightarrow \infty.$
\end{corollary}
We can now establish
 \begin{lemma}\label{1infZ2}
 For every continuous function $\psi$ from $\mathbb{R}_+$ into $\mathbb{R}$ such that $\left|\psi(z) \right| \le C (z^2 \wedge 1)$ for some constant $C,$ we have $$\widetilde{\pi}_{1,N}(\psi) \longrightarrow  \mu (\psi), {~}as{~}N\longrightarrow \infty. $$
\end{lemma}
\begin{proof}
Recalling \eqref{TO}and Corollary \ref{CorpiNbar}, we have that $$\widetilde{\pi}_{1,N}(\varphi)=\bar{\pi}_{N}\left(\frac{\psi}{\bar{f}}\right) \longrightarrow \bar{\pi} \left(\frac{\psi}{\bar{f}}\right)= \mu (\psi), {~}as{~}N\longrightarrow \infty. $$
\end{proof}

$\bf{Proof \ of \  Proposition}$ \ref{PiTildversPiAA}.
 We shall use in two instances the fact that any continuous function $\psi$ with compact support in 
$(0,+\infty)$ satisfies $|\psi(z)|\le C(z^2\wedge1)$, so that we can apply Lemma \ref{1infZ2} to it.

Let $h_M(z)$ be the continuous function from $\R_+$ into $[0,1]$ defined for any $M\ge1$ by
\[ h_M(z)=\begin{cases}  0,&\text{ if $0\le z\le \frac{1}{2M}$};\\
2M(z-\frac{1}{2M}),&\text{ if $\frac{1}{2M}\le z\le\frac{1}{M}$};\\
1,&\text{ if $\frac{1}{M}\le z\le M$};\\
1-M(z-M),&\text{ if $M\le z\le M+1$};\\
0,&\text{ if $z\ge M+1$}. \end{cases}
\]
Let $\varphi$ be an arbitrary continuous function from $\R_+$ into $\R$ such that, we some constant $C>0$, $|\varphi(z)|\le C z$.
It is plain that 
\begin{equation}\label{eq:encadre}
\E\left| \tilde{\pi}_{1,N}(\varphi)-\tilde{\pi}_{1,N}(\varphi h_M)\right|\le\int_0^{1/M}|\varphi(z)|\tilde{\pi}_{1,N}(dz) +\int_M^\infty |\varphi(z)|\tilde{\pi}_{1,N}(dz).
\end{equation}
From Lemma \ref{1infZ2}, for any $M\ge1$, as $N\to\infty$,
\[ \tilde{\pi}_{1,N}(\varphi h_M)\to\mu(\varphi h_M).\]
Next we define for any $M\ge1$ the continuous function $g_M:\R_+\mapsto[0,1]$ by
\[ g_M(z)=\begin{cases}  0,&\text{ if $0\le z\le \frac{1}{M}$};\\
M(z-\frac{1}{M}),&\text{ if $\frac{1}{M}\le z\le\frac{2}{M}$};\\
1,&\text{ if $\frac{2}{M}\le z\le M-1$};\\
1-M(z-M-1),&\text{ if $M-1\le z\le M$};\\
0,&\text{ if $z\ge M$}. \end{cases}
\]
We now note that
\begin{align*}
\int_0^{1/M}|\varphi(z)|\tilde{\pi}_{1,N}(dz) &+\int_M^\infty |\varphi(z)|\tilde{\pi}_{1,N}(dz)\\
&\le C\int_0^{1/M} z\,\tilde{\pi}_{1,N}(dz)+C\int_M^\infty z\,\tilde{\pi}_{1,N}(dz)\\
&\le C\int_0^\infty z\,\mu(dz)-C\int_0^\infty z\, g_M(z)\, \tilde{\pi}_{1,N}(dz).
\end{align*}

Consequently
\begin{equation}\label{eq:restes}
\limsup_{N\to\infty}\left(\int_0^{1/M}|\varphi(z)|\tilde{\pi}_{1,N}(dz) 
+\int_M^\infty |\varphi(z)|\tilde{\pi}_{1,N}(dz)\right)\\
\le C\int_0^\infty z(1-g_M(z))\mu(dz).  
\end{equation}   
   Finally, combining \eqref{eq:encadre} and \eqref{eq:restes}, we obtain   
   \begin{align*}
   &\mu(\varphi h_M)-C\int_0^\infty z(1-g_M(z))\mu(dz)\\   
   &\le\mu(\varphi h_M)-\limsup_{N\to\infty}\left(\int_0^{1/M}|\varphi(z)|\tilde{\pi}_{1,N}(dz) +\int_M^\infty |\varphi(z)|\tilde{\pi}_{1,N}(dz)\right)\\
   &= \liminf_{N\to\infty}\left(\tilde{\pi}_{1,N}(\varphi h_M)
-\int_0^{1/M}|\varphi(z)|\tilde{\pi}_{1,N}(dz) -\int_M^\infty |\varphi(z)|\tilde{\pi}_{1,N}(dz)\right)\\   
&\le\liminf_{N\to\infty} \tilde{\pi}_{1,N}(\varphi)
\le\limsup_{N\to\infty} \tilde{\pi}_{1,N}(\varphi)\\
&\le \limsup_{N\to\infty}\left(\tilde{\pi}_{1,N}(\varphi h_M)
+\int_0^{1/M}|\varphi(z)|\tilde{\pi}_{1,N}(dz) +\int_M^\infty |\varphi(z)|\tilde{\pi}_{1,N}(dz)\right)\\
&=\mu(\varphi h_M)+\limsup_{N\to\infty}\left(\int_0^{1/M}|\varphi(z)|\tilde{\pi}_{1,N}(dz) +\int_M^\infty |\varphi(z)|\tilde{\pi}_{1,N}(dz)\right)\\
&\le \mu(\varphi h_M)+C\int_0^\infty z(1-g_M(z))\mu(dz).
\end{align*}
The result follows by taking the limit as $M\to\infty$, thanks to the dominated convergence theorem, in the 
system of inequalities
 \begin{align*}
   \mu(\varphi h_M)-C\int_0^\infty z(1-g_M(z))\mu(dz) 
   &\le\liminf_{N\to\infty} \tilde{\pi}_{1,N}(\varphi)\\
&\le\limsup_{N\to\infty} \tilde{\pi}_{1,N}(\varphi)\\
&\le \mu(\varphi h_M)+C\int_0^\infty z(1-g_M(z))\mu(dz).
\end{align*}

 $\hfill \blacksquare$
 
Thanks to these results, we are now in position to study the asymptotic properties of the exploration process.
\subsection{ Tightness and Weak convergence of the Contour process}
Consider $\{H_{s}^N, \ s\geq0\}$, the contour process of the forest of trees representing the population $\{Z_{t}^{N,x}, \ t\geq 0\}$. We define $L_{s}^{N}(t)$, the  (scaled) local time accumulated by $H^{N}$ at level $t$ up to time $s$, as
\begin{equation}\label{TL}
 L_{s}^{N}(t) = \frac{2}{c} \lim_{\varepsilon\mapsto 0}\frac{1}{\varepsilon} \int_{0}^{s}\mathbf{1}_{ \{t\le H_{r}^{N} < t+ \varepsilon \}}dr,
\end{equation}
The motivation of the factor  $2/c$ will be clear after we have taken the limit as $N \rightarrow  +\infty$. $L_{s}^{N}(t)$ equals $2 / c$ times the number of pairs of $t$-crossings of $H^{N}$  between times $0$ and $s$. 
Note that this process is neither right- nor left-continuous as a function of $s$. 

Now, we will need to write precisely the evolution of  $\{H_{s}^{N}, \ s\geq0\}.$ To this end, we first recall \eqref{Ptheta} and \eqref{GamaN1N2}. Let $\{\Lambda_k^{1,N}, \ k\geqslant1 \}$ be a sequence of i.i.d r.v's with as joint law of $\Lambda^{1,N}.$ Let $\{{P}_{s}^{1,N}, \ s\geq 0\}$ be a Poisson process, with intensity $2N \gamma_{1,N}.$ We assume that the two sequences $\{{P}_{s}^{1,N}, \ s\geq 0\}$ and $\{\Lambda_k^{1,N}, \ k\geqslant1 \}$ are independent.  Let $\{S_{k}^{1,N}, \ k\geqslant1\}$ be the sequence of successive jump times of the Poisson process $\{{P}_{s}^{1,N}, \ s\geq 0\}.$  Since $\{\Lambda_k^{1,N}, \ k\geqslant1 \}$ is independent of $\{S_{k}^{1,N}, \ k\geqslant1\},$ it follows from Corollary 3.5, p.265 in \cite{ccinlar2011probability} that $\{(\Lambda_k^{1,N},S_{k}^{1,N}), \ k\geqslant1\}$ forms a Poisson random measure $\Pi^{1,N}$ on $(0, \infty)\times \mathbb{N}$ with mean measure $2N\gamma_{1,N}ds \pi_{1,N}(dz)$, recall that $\pi_{1,N}(\{k\})= \P(\Lambda^{1,N}=k)$. More precisely
\begin{remark}\label{MAPPP}
For fixed $s$, we have 
\begin{equation*}
\sum_{k=1}^{\infty} \Lambda_k^{1,N} \mathbf{1}_{(0,s]}\circ S_{k}^{1,N}= \int_{0}^{s} {\Lambda}_{P_{r}^{1,N}}^{1,N}dP_{r}^{1,N}= \int_{[0,s]\times \mathbb{N}} z \Pi^{1,N}(dr,dz)
\end{equation*}
\end{remark}
This implies that one can write 
\begin{equation}\label{Thiabi}
\frac{1}{N}\int_{[0,s]\times \mathbb{N}} z \Pi^{1,N}(dr,dz)= \int_{[0,s]\times \mathbb{R}_+} z \widetilde{\Pi}^{1,N}(dr,dz)
\end{equation}
where $\widetilde{\Pi}^{1,N}$ is a Poisson random measures on $\mathbb{R}_+^{2}$ with mean measures $2 ds \widetilde{\pi}_{1,N}(dz),$ recall that $\widetilde{\pi}_{1,N}$ was defined in \eqref{GamaN1N2}.
 
 Let $\{{P}_{s}^{\prime,1,N}, \ s\geq 0\},$ $\{{P}_{s}^{N}, \ s\geq 0\}$  and $\{{P}_{s}^{\prime,N}, \ s\geq 0\}$ be three mutually independent Poisson processes, with respective intensities $2Nd_{1,N}q_0^{1,N},$ $2N\lambda_N$ and $2N\mu_N.$ We assume that the four processes $\{{P}_{s}^{1,N}, \ s\geq 0\},$  $\{{P}_{s}^{\prime,1,N}, \ s\geq 0\},$ $\{{P}_{s}^{N}, \ s\geq 0\}$  and $\{{P}_{s}^{\prime,N}, \ s\geq 0\}$ are independent. Let us define $${P}_{s}^{N,-}= {P}_{s}^{\prime,1,N}+{P}_{s}^{\prime,N}\quad \mbox{and} \quad {P}_{s}^{N,+}= {P}_{s}^{1,N}+{P}_{s}^{N},\quad \forall \ s\ge0.$$ It is well known that the random process ${P}^{N,-}$ and ${P}^{N,+}$ are Poisson processes, with intensity $2N(d_{1,N}q_0^{1,N}+\mu_N)$ and $2N(\gamma_{1,N}+\lambda_N)$ respectively. Let $\{V_{s}^N, \ s\geq 0\}$  be the c?àdlà?g $\{-1,1\}$-valued process which  is such that, $s-$almost everywhere,  ${dH_{s}^{N}}/{ds}=2 a_NV_{s}^N,$ with $a_N= N+c^{-1}d_{1,N}q_0^{1,N}.$ The $({\mathbb{R}}_{+} \times \{-1,1\})$-valued process  $\{(H_{s}^{N,}, V_{s}^N), \ s\geq 0\}$ solves the SDE
\begin{align}\label{DEFHNVN}
H_{s}^{N}&=2 a_N\int_{0}^{s}V_{r}^{N}dr,  \nonumber\\
V_{s}^{N}&=1+ 2\int_{0}^{s}\mathbf{1}_{\{V_{r^{-}}^{N}=-1\}}dP_{r}^{N,+} -2\int_{0}^{s}\mathbf{1}_{\{V_{r^{-}}^{N}=+1\}}dP_{r}^{N,-}+ cN \big(L_{s}^{N}(0)- L_{0^{+}}^{N}(0)\big) \nonumber\\ 
&+ 2N \sum_{k>0, S_{k}^{N,+}\leq s} \bigg( \frac{c}{2} \Big(L_{s}^{N}(H_{S_{k}^{N,+}}^{N})-L_{S_{k}^{N,+}}^{N}(H_{S_{k}^{N,+}}^{N})\Big)\bigg)\wedge \frac{ (\Lambda_k^{1,N}-1)}{N},
\end{align}
where the $S_{k}^{N,+}$ are the successive jump times of the process 
\begin{equation}\label{PN+}
\mathcal{P}_{s}^{N,+}= \int_{0}^{s}\mathbf{1}_{\{V_{r^{-}}^{N}=-1\}}dP_{r}^{1,N}.
\end{equation}
For any $k>0$, $\Lambda_k^{1,N}-1$ denotes the number of reflections of $H^N$ above the level $H_{S_{k}^{+}}^N$. We deduce from \eqref {DEFHNVN}
\begin{equation*}
\frac{V_{s}^{N}}{2cN}= \frac{1}{2cN} + K_s^N+\frac{1}{cN}\int_{0}^{s}\mathbf{1}_{\{V_{r^{-}}^{N}=-1\}}dP_{r}^{N}-\frac{1}{cN}\int_{0}^{s}\mathbf{1}_{\{V_{r^{-}}^{N}=+1\}}dP_{r}^{N,-}+ \frac{1}{2}\big(L_{s}^{N}(0)- L_{0^{+}}^{N}(0)\big),
\end{equation*}
where
\begin{align}\label{KNSDEF}
K_s^N&= \frac{1}{cN}\int_{0}^{s}\mathbf{1}_{\{V_{r^{-}}^{N}=-1\}}dP_{r}^{1,N} + \frac{1}{c} \sum_{k>0, S_{k}^{N,+}\leq s} \bigg( \frac{c}{2} \Big(L_{s}^{N}(H_{S_{k}^{N,+}}^{N})-L_{S_{k}^{N,+}}^{N}(H_{S_{k}^{N,+}}^{N})\Big)\bigg)\wedge \frac{ ({\Lambda}_{k}^{1,N}-1)}{N}\nonumber\\
& = \frac{1}{cN}\int_{0}^{s} \left[1+\frac{Nc}{2}(L_{s}^{N}(H_{r}^{N})-L_{r}^{N}(H_{r}^{N}))\wedge ({\Lambda}_{\mathcal{P}_{r}^{N,+}}^{1,N}-1)\right] d\mathcal{P}_{r}^{N,+}\nonumber\\
& =   K_{s}^{1,N} - K_{s}^{2,N}, 
\end{align}
with
\begin{align*}
 \quad K_{s}^{1,N}=  \frac{1}{cN} \int_{0}^{s} {\Lambda}_{\mathcal{P}_{r^{-}}^{N,+}+1}^{1,N}d\mathcal{P}_{r}^{N,+} \quad \mbox{ and} 
\end{align*}
\begin{equation*}\label{KN2S}
K_{s}^{2,N} =\frac{1}{cN} \int_{0}^{s} \left({\Lambda}_{\mathcal{P}_{r^{-}}^{N,+}+1}^{1,N} -1- \frac{cN}{2}(L_{s}^{N}(H_{r}^{N})-L_{r}^{N}(H_{r}^{N})) \right)^{+} d\mathcal{P}_{r}^{N,+}.
\end{equation*}
Observe that  $\mathcal{P}_{s}^{N,+}=\mathcal{P}_{s^-}^{N,+}+1$, for $d\mathcal{P}_{s}^{N,+}$ almost every $s$.  Writing the first line of \eqref {DEFHNVN} as
\begin{equation*}
 H_{s}^{N}= 2a_N\int_{0}^{s}\mathbf{1}_{\{V_{r}^{N}=+1\}}dr - 2a_N\int_{0}^{s}\mathbf{1}_{\{V_{r}^{N}=-1\}}dr,
\end{equation*}
denoting by $ \mathcal{M}^{1,N}$, $ \mathcal{M}^{N}$, $\widetilde{\mathcal{M}}^{1,N}$ and $\widetilde{\mathcal{M}}^{N}$ the four local martingales
\begin{equation}\label{M1+}
\mathcal{M}_s^{1,N} =   K_{s}^{1,N} - \frac{1}{cN}  \int_{0}^{s}\mathbf{1}_{\{V_{r}^{N}=-1\}} {{\Lambda}_{{\mathcal{P}_{r}^{1,N}+1}}^{1,N}} \Big( 2N\gamma_{1,N}\Big)dr, 
\end{equation}
\begin{equation}\label{M1-}
\mathcal{M}_s^{N} =  \frac{1}{cN}  \int_{0}^{s}\mathbf{1}_{\{V_{r^{-}}^{N}=-1\}} \Big(dP_r^N -2N\lambda_{N}dr\Big),
\end{equation}
\begin{equation}\label{M2+}
\widetilde{\mathcal{M}}_s^{1,N} =   \frac{1}{cN} \int_{0}^{s}\mathbf{1}_{\{V_{r^{-}}^{N}=+1\}}\left(d{P}_{r}^{\prime,1,N}-2N d_{1,N}q_0^{1,N} dr\right),
\end{equation}
and
\begin{equation}\label{M2-}
\widetilde{\mathcal{M}}_s^{N} =   \frac{1}{cN} \int_{0}^{s}\mathbf{1}_{\{V_{r^{-}}^{N}=+1\}}\left(d{P}_{r}^{\prime,N}-2N\mu_N dr\right),
\end{equation}
and recalling  \eqref{KNSDEF}, we deduce from \eqref{DEFHNVN} 
\begin{align}\label{2DFHN}
 H_{s}^{N}+\frac{V_{s}^{N}}{2cN}=& \frac{1}{2cN}+\mathcal{M}_s^{1,N} + \mathcal{M}_s^{N}-\widetilde{\mathcal{M}}_s^{1,N}-\widetilde{\mathcal{M}}_s^{N} - K_{s}^{2,N}+ \frac{1}{2}(L_{s}^{N}(0)- L_{0^{+}}^{N}(0))\nonumber \\
 &+ \Phi_s^{1,N}+ J^N(s),
\end{align}
with
\begin{equation}\label{MNa} 
 \Phi_s^{1,N}=  \frac{1}{cN}  \int_{0}^{s}\mathbf{1}_{\{V_{r}^{N}=-1\}} \Big(  {{\Lambda}_{{\mathcal{P}_{r}^{1,N}+1}}^{1,N}} -m_{1,N} \Big)\Big(2N\gamma_{1,N}\Big)dr,
\end{equation}
and 
\begin{equation}\label{UNS}
J^N(s)= \frac{2\alpha}{c}  \int_{0}^{s}\mathbf{1}_{\{V_{r}^{N}=-1\}} dr- \frac{2\beta}{c}  \int_{0}^{s}\mathbf{1}_{\{V_{r}^{N}=+1\}} dr.
\end{equation}

The first following statement is elementary.
\begin{lemma}\label{UNSlem}
The sequence $\{J^N, \  N\ge1 \}$ is tight in  $\mathcal{C}([0,\infty)).$
\end{lemma}
Moreover, since $\mathcal{M}_s^{N},$ $\widetilde{\mathcal{M}}_s^{1,N}$ and $\widetilde{\mathcal{M}}_s^{N}$  are purely discontinuous local martingales, we deduce from \eqref{M1-}, \eqref{M2+} and \eqref{M2-} that 
\begin{equation*}
\!\left[\mathcal{M}^{N}\right]_s\!=\!\frac{1}{c^2N^2}\! \int_{0}^{s}\!\mathbf{1}_{\{V_{r}^{N}=-1\}} \!dP_r^N, \left[\widetilde{\mathcal{M}}^{1,N}\right]_s\!=\!\frac{1}{c^2N^2} \int_{0}^{s}\! \mathbf{1}_{\{V_{r}^{N}=+1\}}\!d{P}_{r}^{\prime,1,N}, \ \left[\widetilde{\mathcal{M}}^{N}\right]_s\!=\!\frac{1}{c^2N^2}\! \int_{0}^{s}\! \mathbf{1}_{\{V_{r}^{N}=+1\}}\!d{P}_{r}^{\prime,N},
\end{equation*}
\begin{equation*}
\langle{\mathcal{M}^{N}\rangle}_{s}\!=\!\frac{2\lambda_N}{c^2N}\! \int_{0}^{s}\! \mathbf{1}_{\{V_{r}^{N}=-1\}}\!dr, \langle{\widetilde{\mathcal{M}}^{1,N}\rangle}_{s}\!=\!\frac{2d_{1,N}q_0^{1,N}}{c^2N} \int_{0}^{s}\! \mathbf{1}_{\{V_{r}^{N}=+1\}}\!dr, \langle{\widetilde{\mathcal{M}}^{N}\rangle}_{s}\!=\!\frac{2 \mu_N}{c^2N} \int_{0}^{s}\! \mathbf{1}_{\{V_{r}^{N}=+1\}}\!dr.
\end{equation*}
It follows readily that there exists a constant $C>0$ such that $$\mathbb{E}(\langle \mathcal{M}^{N}\rangle_T) \le CT, \quad \mbox{for all} \ T>0, \ N\ge1.$$ Hence $\mathcal{M}^{N}$ is a square integrable martingale. We can prove similarly that $\widetilde{\mathcal{M}}^{N}$ is a square integrable martingale. However, we deduce from the proof of Lemma \ref{d1Nq01N} that there exists a constant $C>0$ such that 
\begin{equation}\label{mart}
\mathbb{E}\left(\langle\widetilde{\mathcal{M}}^{1,N}\rangle_T\right) \le \frac{2}{c^2N}T \int_{0}^{\infty} z\ \mu(dz), \quad \mbox{for all} \ T>0.
\end{equation}
\begin{corollary}
  $\{ \mathcal{M}_s^{N}, \ s\geq0\},$ $\{\widetilde{\mathcal{M}}_s^{1,N}, \ s\geq0\}$ and $\{\widetilde{\mathcal{M}}_s^{N}, \ s\geq0\}$ are square integrable martingales.
\end{corollary}
We shall need below the
\begin{lemma}\label{Martingale}
There exists a constant $C$ such that for all $T>0$,
\begin{equation*}
\mathbb{E}\left(\sup_{0\leq s\leq T} \left|\widetilde{\mathcal{M}}_s^{1,N}\right| \right) \leq \frac{C}{N}T, \quad \mathbb{E}\left(\sup_{0\leq s\leq T} \left|{\mathcal{M}}_s^{N}\right| \right) \leq CT \quad \mbox{and} \quad \mathbb{E}\left(\sup_{0\leq s\leq T} \left|\widetilde{\mathcal{M}}_s^{N}\right| \right) \leq CT.
\end{equation*}
recall that $C$ denote a constant which may differ from line to line.
\end{lemma}
\begin{proof}
We have
\begin{equation*}
\mathbb{E}\left(\sup_{0\leq s\leq T} \left|\widetilde{\mathcal{M}}_s^{1,N}\right| \right) \leq  \left[\mathbb{E}\left(\sup_{0\leq s\leq T}   \left| \widetilde{\mathcal{M}}_s^{1,N}\right| \right)^2\right]^{\frac{1}{2}}.
\end{equation*}
This together with \eqref{mart}, Doob's $L^2$-inequality for martingales implies 
\begin{equation*}
\mathbb{E}\left(\sup_{0\leq s\leq T} \left|\widetilde{\mathcal{M}}_s^{1,N}\right| \right) \leq \frac{C}{N}T. 
\end{equation*}
We can prove similarly that 
\begin{equation*}
 \mathbb{E}\left(\sup_{0\leq s\leq T} \left|{\mathcal{M}}_s^{N}\right| \right) \leq CT \quad \mbox{and} \quad \mathbb{E}\left(\sup_{0\leq s\leq T} \left|\widetilde{\mathcal{M}}_s^{N}\right| \right) \leq CT.
\end{equation*}
\end{proof}

From this proof, we deduce the following  
\begin{corollary}\label{ConvM}
As $N\longrightarrow\infty$, $$\Big(\widetilde{\mathcal{M}}_s^{1,N}, \ s\geq0\Big)\longrightarrow  0 \ \mbox{ in} \ \mbox{probability}, \ \mbox{locally} \ \mbox{uniformly} \ \mbox{in} \ \mbox{s}.$$ 
\end{corollary}

 It is easy to obtain from Remarks \ref{MAPPP}, \eqref{Thiabi} and  \eqref{KNSDEF} that  
\begin{align}\label{KN2SS}
K_s^{2,N}= \frac{1}{c} \int_{0}^{s}  \int_{0}^{\infty} \mathbf{1}_{\{V_{r^{-}}^{N}=-1\}} \left(z-\frac{1}{N}- \frac{c}{2}(L_{s}^{N}(H_{r}^{N})-L_{r}^{N}(H_{r}^{N})) \right)^{+}  \widetilde{\Pi}^{1,N}(dr,dz). 
\end{align}
 Recall that $\widetilde{\Pi}^{1,N}$ is a Poisson random measures on $\mathbb{R}_+^{2}$ with mean measures $2 ds \widetilde{\pi}_{1,N}(dz).$
 
We shall need below the 
\begin{lemma}\label{LemporK2}
There exist a constant $C>0$ such that for all $T>0$,
$$\E\left(\sup_{0\le s\le T} K_s^{2,N}\right)\le CT.$$
\end{lemma}
\begin{proof}
we first note that 
\begin{equation*}
\sup_{0\le s\le T} K_{s}^{2,N} \le \frac{1}{c} \int_{0}^{T}  \int_{0}^{\infty} z \widetilde{\Pi}^{1,N}(dr,dz). 
\end{equation*}
By taking the limit on both side, we obtain 
\begin{equation*}
\E\left(\sup_{0\le s\le T} K_{s}^{2,N} \right) \le \frac{2}{c} T  \int_{0}^{\infty} z \widetilde{\pi}_{1,N}(dz). 
\end{equation*}
Hence the Lemma follows readily from Lemma \ref{d1Nq01N} and assumption ${(\bf H)}$. 
\end{proof}

For $s>0$, define 
$$ \mathcal{P}_{s}^{N,-}= \int_{0}^{s}\mathbf{1}_{\{V_{r^{-}}^{N}=+1\}}d{P}_{r}^{\prime,1,N}.$$ 
Let $(\lambda_{1,N}(s), s\geq0)$ $[$ resp. $(\lambda_{2,N}(s), s\geq0)]$ denote the intensity of the process $(\mathcal{P}_{s}^{N,+}, s\geq0)$ $[$ resp. $(\mathcal{P}_{s}^{N,-}, s\geq0)]$ where $\mathcal{P}^{N,+}$ was defined in \eqref{PN+}. In other words, we have $$\lambda_{1,N} (r)=2N\gamma_{1,N}\mathbf{1}_{\{V_{r^{-}}^{N}=-1\}}\quad \mbox{and} \quad \lambda_{2,N} (r)=2Nd _{1,N} q_0^{1,N}\mathbf{1}_{\{V_{r^{-}}^{N}=+1\}}.$$
For the rest of this subsection we set 
$$ A_{s}^{1,N}=\int_{0}^{s}\lambda_{1,N} (r)dr, \quad A_{s}^{2,N}=\int_{0}^{s}\lambda_{2,N}(r) dr,$$ 
$$\Delta_{s}^{1,N}=N\gamma_{1,N}s,  \quad \Delta_{s}^{2,N}=Nd_{1,N}q_0^{1,N}s  \quad \mbox{and} \quad \overline{\Lambda}_{k}^{1,N}=    {\Lambda}_{k}^{1,N} -m_{1,N}. $$
\begin{remark}\label{ChangeTime}
Note that $\mathcal{P}^{N,+}$ ( resp. $\mathcal{P}^{N,-}$)  can be viewed as time-changed of mutually independent standard Poisson processes $P_{1}$ (resp. $P_{2}$) i.e. $$ \mathcal{P}_{s}^{N,+}=P_{1}\left(A_{s}^{1,N}\right), \ and \ 
 \mathcal{P}_{s}^{N,-}=P_{2}\left(A_{s}^{2,N}\right).$$
\end{remark}
From \eqref{MNa} and Remark \ref{ChangeTime}, we deduce that
\begin{align*}
 \Phi_s^{1,N}&=  \frac{1}{cN}\int_{0}^{s} { \overline{\Lambda}_{{P_1(A_{r}^{1,N})+1}}^{1,N}}dA_{r}^{1,N}= \frac{1}{cN}\int_{0}^{A_{s}^{1,N}}{ \overline{\Lambda}_{{P_1(u)+1}}^{1,N}}du \nonumber \\ 
&= \frac{1}{cN} \bigg(\sum_{k=0}^{P_1(A_{s}^{1,N})} {\overline{\Lambda}_{k+1}^{1,N}} \Xi_k-{ \overline{\Lambda}_{{P_1(A_{s}^{1,N})+1}}^{1,N}}({T}^+(A_{s}^{1,N}) - A_{s}^{1,N})\bigg),
\end{align*}
where $ \Xi_k$ denotes the length of the time interval during which $P_1(u)=k$ and  ${T}^+(A_{s}^{1,N})$ is the first jump time of $P_1$ after $A_{s}^{1,N}$. It is easily seen that $\Xi_k$ has the standard exponential distribution and we notice that  $({\Xi}_1, {\Lambda}_{1}^{1,N}, {\Xi}_2, {\Lambda}_{2}^{1,N},\cdot \cdot \cdot)$ is a sequence of independent random variables. By the same computations, we deduce from \eqref{M1+} 
\begin{equation*}
\mathcal{M}_s^{1,N} = -\frac{1}{cN} \bigg(\sum_{k=0}^{P_1(A_{s}^{1,N})} { {\Lambda}_{k+1}^{1,N}}  \overline{\Xi}_k- {{\Lambda}_{{P_1(A_{s}^{1,N})+1}}^{1,N}}({T}^+(A_{s}^{1,N}) - A_{s}^{1,N})\bigg),
\end{equation*}
where $ \overline{\Xi}_k= {\Xi}_k- \E({\Xi}_k)$.  From \eqref{M2+}, we have also
\begin{equation*}
\widetilde{\mathcal{M}}_s^{1,N}= -\frac{1}{cN} \bigg(\sum_{k=0}^{P_2({A}_{s}^{2,N})}\overline{\Xi}^{\prime}_k- ({T}^-({A}_{s}^{2,N}) - {A}_{s}^{2,N})\bigg),
\end{equation*}
where $ \overline{\Xi}^{\prime}_k= {\Xi}^{\prime}_k - \E({\Xi}^{\prime}_k)$ and where ${\Xi}^{\prime}_k$ denotes the length of the time interval during which $P_2(u)=k$ and  ${T}^-({A}_{s}^{2,N})$ is the first jump time of $P_2$ after ${A}_{s}^{2,N}$. As previously ${\Xi}^{\prime}_k$ has the standard exponential distribution and $({\Xi}^{\prime}_1,{\Xi}_1, {\Lambda}_{1}^{1,N}, {\Xi}^{\prime}_2,{\Xi}_2,{\Lambda}_{2}^{1,N},\cdot \cdot \cdot)$ is a sequence of independent random variables.
  
We notice that 
$$\mathcal{M}_s^{1,N}+ \Phi_s^{1,N}=  \frac{1}{cN} \sum_{k=0}^{P_1(A_{s}^{1,N})}( {\Lambda}_{k+1}^{1,N}-m_{1,N}{\Xi}_k) +\frac{m_{1,N}}{cN} \left({T}^+(A_{s}^{1,N}) - A_{s}^{1,N}\right).$$
 If we define for  $\ell \ge1$ 
 \begin{equation}\label{SnetPHiN}
  S_{\ell}^{1,N}=\sum_{k=0}^{\ell} ( {\Lambda}_{k+1}^{1,N}-m_{1,N}{\Xi}_k)  \quad \mbox{and} \quad   S_{\ell}=\sum_{k=0}^{\ell} \overline{\Xi}^{\prime}_k, 
\end{equation} 
we obtain the following relations
\begin{equation}\label{MN1+MNa}
\mathcal{M}_s^{1,N}+ \Phi_s^{1,N}=\frac{1}{cN} S_{P_1(A_{s}^{1,N})}^{1,N} + \mathbf{Q}^+(A_{s}^{1,N}), 
\end{equation}
 with 
 \begin{equation*}
\mathbf{Q}^+(A_{s}^{1,N})= \frac{m_{1,N}}{cN} \left({T}^+(A_{s}^{1,N}) - A_{s}^{1,N}\right), 
 \end{equation*}
and
\begin{equation}\label{SECONMDEU}
\widetilde{\mathcal{M}}_s^{1,N}= -\frac{1}{cN} S_{P_2(A_{s}^{2,N})} + \mathbf{Q}^-(A_{s}^{2,N})
\end{equation}
with 
$$\mathbf{Q}^-(A_{s}^{2,N}) = \frac{1}{cN}\left({T}^-(A_{s}^{2,N}) - A_{s}^{2,N}\right).$$
We shall need below the  
\begin{lemma}\label{Q+Q-}
As $N\longrightarrow\infty$,  $\mathbf{Q}^+(A_{s}^{1,N})[resp.\ \mathbf{Q}^-(A_{s}^{2,N}) ]\longrightarrow 0 \ \mbox{ in} \ \mbox{probability}, \ \mbox{locally} \ \mbox{uniformly} \ \mbox{in} \ \mbox{s} $.
\end{lemma}
\begin{proof}
The proof follows an argument similar to the proof of Lemma 4.24 in  \cite{drame2016non}.
\end{proof}

Moreover, we have
\begin{lemma}\label{Q1+}
There exist a constant $C>0$ such that for all $T>0$,
$$\E\left(\sup_{0\le s\le T} \mathbf{Q}^+(A_{s}^{1,N})\right)\le CT.$$
\end{lemma}
\begin{proof}
We notice that 
$$\big|\mathbf{Q}^+(A_{s}^{1,N})\big|\le \frac{m_{1,N}}{cN} \Xi_{P(A_{s}^{1,N})+1},$$ 
this implies 
\begin{align*}
\sup_{0\le s\le T} \big|\mathbf{Q}^+(A_{s}^{1,N})\big|&\le  \frac{m_{1,N}}{cN}  \sup_{0\le k\le P(A_{s}^{1,N})}\Xi_{k+1}\\
&\le  \frac{m_{1,N}}{cN}  \sup_{0\le k\le P(2\Delta_{T}^{1,N})}\Xi_{k+1}\\
&\le  \frac{m_{1,N}}{cN}  \sum_{k=0}^{P(2\Delta_{T}^{1,N})}\Xi_{k+1}
\end{align*}
Hence taking expectation in both side and using wald's identity, we deduce that
\begin{align*}
\E\left( \sup_{0\le s\le T} \big|\mathbf{Q}^+(A_{s}^{1,N})\big| \right)&\le  \frac{m_{1,N}}{cN} \E(P(2\Delta_{T}^{1,N})) \E(\Xi_{1})\\
&=  \frac{2m_{1,N} \gamma_{1,N}}{c} T \\
&=\frac{2}{cN}T\int_{0}^{\infty}z\widetilde{\pi}_{1,N}(dz)
\\& \le CT.
\end{align*}
where we have used the same arguments as in the proof of Lemma \ref{d1Nq01N}. Hence the desired result follows.
\end{proof}

We next define for $\ell \ge1$ 
 \begin{equation}\label{PHINFN}
\Upsilon_{\ell}^N= \frac{1}{cN} S_{\ell}^{1,N} \quad \mbox{and} \quad   F_s^N=  \Upsilon_{P_1(A_{s}^{1,N})}^{N}.
\end{equation} 
Recalling \eqref{MN1+MNa} and let us rewrite \eqref{2DFHN} in the form
\begin{align}\label{2DFHN22}
 H_{s}^{N}+\frac{V_{s}^{N}}{2cN}=& \frac{1}{2cN}+F_s^N+\mathbf{Q}^+(A_{s}^{N,+}) + \mathcal{M}_s^{N}-\widetilde{\mathcal{M}}_s^{1,N}-\widetilde{\mathcal{M}}_s^{N} \nonumber\\
 &+ J^N(s)- K_{s}^{N,2}+ \frac{1}{2}(L_{s}^{N}(0)- L_{0^{+}}^{N}(0)).
 \end{align}
We want to estimate the stochastic process $F^N.$ Thus, we first need to prove some intermediate results.To this end, let us define $X^{1,N}= {\Lambda}_{1}^{1,N}-m_{1,N}{\Xi}_1.$ Recall that $p$ is fixed real number with $1<p<2$. In what follows $X^{1,N}$ will be a random variable satisfying $\E (X^{1,N})=0$, $\E (X^{1,N})^2=\infty$, and $X_1^{1,N},X_2^{1,N},\cdots$ will be independent random variables each having the distribution of $X^{1,N}.$ $C_1,C_2,\cdots$ will be positive constants depending only $p.$ However, we deduce from \eqref{Lamda1N} that  
\begin{align*}
X^{1,N}&= \left[\Gamma_N^+ \mathbf{1}_{\{\hat{\epsilon}_N=1\}} - {\Xi}_1 \E \left(\Gamma_N^+ \mathbf{1}_{\{\hat{\epsilon}_N=1\}}\right) \right]+ \left[\Gamma_N^- \mathbf{1}_{\{\hat{\epsilon}_N=0\}}  - {\Xi}_1 \E \left(\Gamma_N^- \mathbf{1}_{\{\hat{\epsilon}_N=0\}}\right)  \right]\\
&=T^{N,+}+T^{N,-}.
 \end{align*}
For $N\ge 1$ define 
\begin{equation*}
W_2(N)= \E\left([T^{N,-}]^2 \right)\quad \mbox{and} \quad \bar{W}_p(N)= \E\left(\left[T^{N,+}\right]^p \right).
\end{equation*}
We first need the two following Lemmas.  
\begin{lemma}\label{Devis1}
There exist $C$  such that 
\begin{equation*}
 W_2(N) \le  \frac{CN}{\gamma_{1,N}},
\end{equation*}
where $\gamma_{1,N}$ was defined in \eqref{GamaN1N2}.
\end{lemma}
\begin{proof}
We first recall that $\Gamma_N^-$ is a random variable with probability distribution $p^{-,N}$ and where $p^{-,N}$ was defined in \eqref{pN-pN+}. Now, we have 
\begin{align}\label{xt}
W_2(N)&=\E\left([T^{N,-}]^2 \right)\nonumber
\\&= \E \left(\left[\Gamma_N^- \mathbf{1}_{\{\hat{\epsilon}_N=0\}}  - {\Xi}_1 \E \left(\Gamma_N^- \mathbf{1}_{\{\hat{\epsilon}_N=0\}}\right)  \right]^2 \right)\nonumber
\\&\le   \E \left[(\Gamma_N^- )^2\mathbf{1}_{\{\hat{\epsilon}_N=0\}} \right] + 2 \left[ \E \left( \Gamma_N^- \mathbf{1}_{\{\hat{\epsilon}_N=0\}}\right)  \right]^2 \nonumber
\\&\le C \E \left[(\Gamma_N^- )^2\mathbf{1}_{\{\hat{\epsilon}_N=0\}} \right] =C \P (\hat{\epsilon}_N=0)  \E (\Gamma_N^- )^2.
\end{align}
However, recalling \eqref{q0-N}, we deduce from \eqref{pN-pN+} that  
\begin{align*}
 \E (\Gamma_N^- )^2&= \sum_{k\ge1} k^2 p_k^{-,N}= \sum_{k\ge1} k^2 \frac{ q_{k+1}^{-,N}}{1- q_0^{-,N}}=  \frac{1}{1- q_0^{-,N}} \sum_{k\ge2} (k-1)^2 q_{k}^{-,N}\\
&\le   \frac{1}{1- q_0^{-,N}} \sum_{k\ge2} k(k-1) q_{k}^{-,N}\\
&=  \frac{1}{N\alpha_{-,N}\left(1- q_0^{-,N}\right)} \sum_{k\ge2}  \frac{k(k-1)}{k! }\int_{0}^{1}(Nr)^k e^{-Nr} \mu(dr)\\
&= \frac{N}{\alpha_{-,N}\left(1- q_0^{-,N}\right)}\int_{0}^{1} \left( \sum_{k\ge2}  \frac{1}{(k-2)! }(Nr)^{k-2}  \right) e^{-Nr}  r^2 \mu(dr)\\
&= \frac{N}{\alpha_{-,N}\left(1- q_0^{-,N}\right)}\int_{0}^{1} r^2 \mu(dr).
 \end{align*}
 Now, combining this with \eqref{xt} and the fact that $$ \P (\hat{\epsilon}_N=0)=\alpha_{-,N}(1- q_0^{-,N})/d_{1,N}(1- q_0^{1,N}),$$ we deduce that 
 \begin{align*}
 W_2(N)&\le  \frac{CN}{d_{1,N}(1- q_0^{1,N})}\int_{0}^{1} r^2 \mu(dr)\\
 &=  \frac{CN}{\gamma_{1,N}}\int_{0}^{1} r^2 \mu(dr).
 \end{align*}
 Hence the Lemma follows readily from assumption ${(\bf H)}$.
\end{proof}
\begin{lemma}\label{Devis1p}
There exist $C_1$  such that 
\begin{equation*}
\bar{W}_p(N) \le  \frac{C_1N^{p-1}}{\gamma_{1,N}}.
\end{equation*}
\end{lemma}
\begin{proof}
In this proof, we shall use the following inequality 
\begin{equation}\label{eq}
 y^{p-1}\le 1+y, \ \mbox{for all}  \ y>0.
\end{equation}
 Recall that $\Gamma_N^+$ is a random variable with probability distribution $p^{+,N}$ and where $p^{+,N}$ was defined in \eqref{pN-pN+}. Now, we have 
 \begin{align}\label{xxt}
\bar{W}_p(N)&=\E\left(\left[T^{N,+}\right]^p \right)\nonumber
\\&= \E \left(\left[\Gamma_N^+ \mathbf{1}_{\{\hat{\epsilon}_N=1\}}  - {\Xi}_1 \E \left(\Gamma_N^+ \mathbf{1}_{\{\hat{\epsilon}_N=1\}}\right)  \right]^p \right)\nonumber
\\&\le 2^{p-1}  \E \left[(\Gamma_N^+ )^p\mathbf{1}_{\{\hat{\epsilon}_N=1\}} \right] + 2^{p-1}\times \Gamma(p+1) \left[ \E \left( \Gamma_N^+ \mathbf{1}_{\{\hat{\epsilon}_N=1\}}\right)  \right]^p \nonumber
\\&\le C \E \left[(\Gamma_N^+ )^p\mathbf{1}_{\{\hat{\epsilon}_N=1\}} \right] =C \P (\hat{\epsilon}_N=1)  \E (\Gamma_N^+ )^p.
\end{align}
However, recalling \eqref{q0+N}, we deduce from \eqref{pN-pN+} that  
\begin{align}\label{etoil}
 \E (\Gamma_N^+ )^p&= \sum_{k\ge1} k^p p_k^{+,N}= \sum_{k\ge1} k^p \frac{ q_{k+1}^{+,N}}{1- q_0^{+,N}}=  \frac{1}{1- q_0^{+,N}} \sum_{k\ge2} (k-1)^p q_{k}^{+,N} \nonumber\\
&\le   \frac{1}{1- q_0^{+,N}} \sum_{k\ge2} k^p q_{k}^{+,N} \nonumber\\
&=  \frac{1}{N\alpha_{+,N}\left(1- q_0^{+,N}\right)} \sum_{k\ge2}  \frac{k^p}{k! }\int_{1}^{\infty}(Nr)^k e^{-Nr} \mu(dr)\nonumber\\
&= \frac{N^p}{N\alpha_{+,N}\left(1- q_0^{+,N}\right)}\int_{1}^{\infty} \left( \sum_{k\ge2}  \left(\frac{k}{Nr} \right)^p \frac{(Nr)^k}{k!}  \right) e^{-Nr}  r^p \mu(dr).
 \end{align}
 However, using \eqref{eq}, we have
 \begin{align*}
  \sum_{k\ge2}  \left(\frac{k}{Nr} \right)^p \frac{(Nr)^k}{k!}  &\le    \sum_{k\ge2}  \left(\frac{k}{Nr} \right) \frac{(Nr)^k}{k!} +   \sum_{k\ge2}  \left(\frac{k}{Nr} \right)^2 \frac{(Nr)^k}{k!}\\
  &=   \sum_{k\ge2}  \frac{1}{(k-1)!} {(Nr)}^{k-1}+ \sum_{k\ge2}  \frac{k}{(k-1)!} {(Nr)}^{k-2}\\
  &\le C e^{Nr}. 
  \end{align*}
 Combining this with \eqref{etoil}, we deduce that  
 \begin{equation*}
  \E (\Gamma_N^+ )^p \le \frac{CN^{p-1}}{\alpha_{+,N}\left(1- q_0^{+,N}\right)}\int_{1}^{\infty} r^p \mu(dr).
 \end{equation*}
  Now, combining this with \eqref{xxt} and the fact that $$ \P (\hat{\epsilon}_N=1)=\alpha_{+,N}(1- q_0^{+,N})/d_{1,N}(1- q_0^{1,N}),$$ we deduce that 
 \begin{align*}
\bar{W}_p(N)&\le  \frac{CN^{p-1}}{d_{1,N}(1- q_0^{1,N})}\int_{1}^{\infty} r^p \mu(dr)\\
 &=  \frac{CN^{p-1}}{\gamma_{1,N}}\int_{1}^{\infty} r^p \mu(dr).
 \end{align*}
 Hence the Lemma follows readily from Assumption ${(\bf H)} .$
\end{proof}

Let $\bf{X}$ be a random variable. For $\ell\ge1$, we define its $\ell$-norm by $\|\bf{X} \|_\ell= \big[\E |\bf{X}|^\ell\big]^{\frac{1}{\ell}}$. The main tool in the following lemma is the existence of positive constants $k_p$ and $K_p$ depending only on $p$ such that for all $\ell \ge 1$, 
 \begin{equation}\label{ZYGMUNG}
 k_p \left\|\left(\sum_{k=0}^{\ell}\left[X_{k}^{1,N}\right]^2 \right)^{\frac{1}{2}} \right\|_p \le  \left\|   S_{\ell}^{1,N} \right\|_p \le  K_p \left\|\left(\sum_{k=0}^{\ell}\left[X_{k}^{1,N}\right]^2 \right)^{\frac{1}{2}} \right\|_p
 \end{equation} 
see Theorem 5 in  \cite{marcinkiewicz1938quelques}. Recall that $S^{1,N}$ was given by $$S_{\ell}^{1,N}=\sum_{k=0}^{\ell} X_{k}^{1,N},$$ where  $\{ X_{k}^{1,N}, k\ge1\}$ is a sequence of i.i.d rv's. Note also that $P_1$ and the sequence $\{ X_{k}^{1,N}, k\ge1\}$ are independent. We now prove
\begin{lemma}\label{Devis2}
There exists a constant $C_2$ such that  for all $T>0$
 \begin{equation*}
\left\|   S_{P_1(2\Delta_{T}^{1,N})}^{1,N} \right\|_p \le C_2 N T. 
 \end{equation*}
\end{lemma}
\begin{proof}
Recalling the inequality 
\begin{equation}\label{inega}
\left(\sum_{i=1}^{k} y_i \right)^{\frac{p}{2}} \le \sum_{i=1}^{k} y_i^{\frac{p}{2}} , \quad \mbox{for all} \ k\ge 2, \ y_1,...,y_k\ge 0,
\end{equation}
and recall also that $X_{k}^{1,N}=T_k^{N,+}+T_k^{N,-}.$ Using the righthand side of \eqref{ZYGMUNG},
\begin{align*}
\left\|   S_{P_1(2\Delta_{T}^{1,N})}^{1,N} \right\|_p &\le K_p \left[  \left\| \left( \sum_{k=0}^{P_1(2\Delta_{T}^{1,N})}\left[X_{k}^{1,N}\right]^2 \right)^{\frac{1}{2}} \right\|_p \right]
\\&\le  K_p \left[  \left\| \left(\sum_{k=0}^{P_1(2\Delta_{T}^{1,N})} \left[T_k^{N,-}\right]^2 \right)^{\frac{1}{2}} \right\|_p +  \left\|  \left( \sum_{k=0}^{P_1(2\Delta_{T}^{1,N})}\left[T_k^{N,+}\right]^2 \right)^{\frac{1}{2}}\right\|_p \right]
\\& \le  K_p \left[  \left\|\left(\sum_{k=1}^{P_1(2\Delta_{T}^{1,N})} \left[T_k^{N,-}\right]^2 \right)^{\frac{1}{2}} \right\|_2 +\left[ \E \left(\sum_{k=1}^{P_1(2\Delta_{T}^{1,N})}\left[T_k^{N,+}\right]^p \right) \right]^{\frac{1}{p}}    \right]
\\& \le  K_p \left[  \left[ \E \left(P_1(2\Delta_{T}^{1,N}) \right) \E\left(\left[T_1^{N,-}\right]^2\right)\right]^{\frac{1}{2}} + \left[ \E \left(P_1(2\Delta_{T}^{1,N}) \right) \E\left(\left[T_1^{N,+}\right]^p\right)\right]^{\frac{1}{p}} \right]
\\& \le  K_p \left[  \left[ \E(P_1(2\Delta_{T}^{1,N}) ) W_2(N)\right]^{\frac{1}{2}} + \left[ \E (P_1 (2\Delta_{T}^{1,N})) \bar{W}_p(N)\right]^{\frac{1}{p}} \right]
\\& \le K_p \left[  \left( CN\gamma_{1,N} T  \frac{CN}{\gamma_{1,N}} \right)^{\frac{1}{2}} + \left( CN\gamma_{1,N} T \frac{C_1N^{p-1}}{\gamma_{1,N}}  \right)^{\frac{1}{p}} \right]
\\&\le C_2 N T,
\end{align*}
where we have used the Minkowski inequality for the second inequality, \eqref{inega} for the 3th inequality, the Wald identity for the 4th inequality and finally Lemmas \ref{Devis1} and \ref{Devis1p} and the fact that 
$\E(P_{1}(2\Delta_{T}^{1,N}))\le CN\gamma_{1,N} T.$
\end{proof}

Lemma \ref{Devis2} combined with \eqref{PHINFN} leads to
\begin{corollary}\label{CORR2}
There exists a constant $C_3$ such that  for all $T>0$
\begin{equation*}
\left\|   \Upsilon_{P_1\left(2\Delta_{T}^{1,N}\right)}^{N} \right\|_p \le C_3  T.
 \end{equation*}
\end{corollary}
We will need the following lemmas
\begin{lemma}\label{SupKN1S}
There exists a constant $C_4$ such that for all $T>0$,
\begin{equation*}
\mathbb{E}\left(\sup_{0\leq s\leq T}\big|{F}_{s}^{N}\big|\right)\leq C_4 T,
\end{equation*}
(recall that $F^{N}$ was defined in \eqref{PHINFN}).
\end{lemma}
\begin{proof}
 We have
\begin{align*}
\mathbb{E}\left(\sup_{0\leq s\leq T} \left|{F}_{s}^{N}\right| \right)&\leq \left[\mathbb{E}\left( \sup_{0\leq k\leq P_1(A_{T}^{1,N})}  \left| \Upsilon_{k}^{N}\right|\right)^p\right]^{\frac{1}{p}}
\\&\leq   \left[\mathbb{E}\left( \sup_{0\leq k\leq P_1(2\Delta_{T}^{1,N})}  \left| \Upsilon_{k}^{N}\right|\right)^p\right]^{\frac{1}{p}}. 
\end{align*}
From \eqref{PHINFN}, it is easy to check that $( \Upsilon_{k}^N, \ k\ge0)$ is a discrete-time martingale. Moreover, note that $P_1(\Delta_{T}^{1,N})$ is a stopping time. Hence, from Doob's inequality we have 
\begin{equation*}
\mathbb{E}\left(\sup_{0\leq s\leq T} \left|{F}_{s}^{N}\right| \right) \leq \frac{p}{p-1} \left\|      \Upsilon_{P_1\left(2\Delta_{T}^{1,N}\right)}^{N} \right\|_p 
\end{equation*}
The result now follows readily from Corollary \ref{CORR2}.
\end{proof}

The following Proposition plays a key role in the asymptotic behaviour of $H^{N}$ 
\begin{proposition}\label{supHN} 
There exist a constant $C>0$ such that for all $T>0$,
$$\E\left(\sup_{0\le s\le T} H_s^N\right)\le CT.$$
\end{proposition}
\begin{proof}
Let us rewrite \eqref{2DFHN22} in the form 
\begin{equation}\label{HsimaN}
H_{s}^{N}= \Sigma_{s}^{N} +\frac{1}{2}L_{s}^{N} (0),
\end{equation}
where 
\begin{align}\label{WNSTB}
\Sigma_{s}^{N}&=  \frac{1}{2cN}-\frac{V_{s}^{N}}{2cN}+F_s^N+\mathbf{Q}^+(A_{s}^{N,+}) + \mathcal{M}_s^{N}-\widetilde{\mathcal{M}}_s^{1,N}-\widetilde{\mathcal{M}}_s^{N} + J^N(s)-K_{s}^{N,2}- \frac{1}{2} L_{0^{+}}^{N}(0)\nonumber\\
&\le F_s^N+\mathbf{Q}^+(A_{s}^{N,+}) + \mathcal{M}_s^{N}-\widetilde{\mathcal{M}}_s^{1,N}-\widetilde{\mathcal{M}}_s^{N} + J^N(s)-K_{s}^{N,2}.
\end{align}
since, from \eqref{TL}, it is easily checked that $L_{0^{+}}^{N}(0)= 2/cN$. Set $$s_r= \sup \left\{ 0\le r\le s;  \ L_{r}^{N}(0)- L_{r^{-}}^{N}(0)>0\right\},$$
then the fact that $L^{N}(0)$ is  increasing, and increases only on the set of time when $H_{s}^{N}=0$ proves that  $H_{s_r}^{N}=0$ and $L_{s}^{N}(0)= L_{s_r}^{N}(0)$. Hence, we obtain that  
\begin{align*}
H_{s}^{N}&= \Sigma_{s}^{N}- \Sigma_{s_r}^{N} \nonumber\\
&\le \sup_{0\le r \le s} [\Sigma_{s}^{N}- \Sigma_{r}^{N}].
 \end{align*}
It follows that 
\begin{equation}\label{SU}
\sup_{0\le s \le T} H_{s}^{N}  \le 2 \sup_{0\le s \le T} |\Sigma_{s}^{N}|.
\end{equation}
Since, moreover $(K_s^{N,2},\ s\ge0)$ is a process with values in ${\mathbb{R}}_{+}$ see \eqref{KN2SS}, we have from \eqref{WNSTB} that 
\begin{align*}
\sup_{0\le s \le T} | \Sigma_{s}^{N} | &\le \sup_{0\leq s\leq T} |F_{s}^{N}|  \! +\!\sup_{0\leq s\leq T} \mathbf{Q}^+(A_{s}^{N,+})\!+ \!\sup_{0\leq s\leq T} |\mathcal{M}_{s}^{N}| \!+\! \sup_{0\leq s\leq T} |\widetilde{\mathcal{M}}_s^{1,N}| +\! \sup_{0\leq s\leq T} |\widetilde{\mathcal{M}}_s^{N}|\\
&+ \sup_{0\leq s\leq T} |J^N(s) |+\sup_{0\leq s\leq T} K_s^{2,N} 
\end{align*}
Combining this inequality with  \eqref{SU}, we deduce that 
\begin{align*}
\sup_{0\le s\le T} H_{s}^{N}\! \le&  2 \sup_{0\le s\le T} \! | F_{s}^{N}| \!+ \! 2\!\sup_{0\le s\le T}\! \mathbf{Q}^+(A_{s}^{N,+})\!+2\!\sup_{0\leq s\leq T} |\mathcal{M}_{s}^{N}| \!+2\! \sup_{0\leq s\leq T} |\widetilde{\mathcal{M}}_s^{1,N}| +2\! \sup_{0\leq s\leq T} |\widetilde{\mathcal{M}}_s^{N}| \!\\& +2 \!\sup_{0\leq s\leq T} |J^N(s) | +2\sup_{0\leq s\leq T} K_s^{2,N}.
\end{align*}
 Hence taking expectation in both side, we deduce that 
 \begin{align*}
\E\left(\sup_{0\le s\le T} H_{s}^{N}\right) \le&  2 \E\left(\sup_{0\le s\le T} \! | F_{s}^{N}| \right)+  2\E\left(\sup_{0\le s\le T} \mathbf{Q}^+(A_{s}^{N,+})\right)+2 \E\left( \sup_{0\leq s\leq T} |\mathcal{M}_{s}^{N}| \right) +2 \E\left( \sup_{0\leq s\leq T} |\widetilde{\mathcal{M}}_s^{1,N}| \right) \\&+2 \E\left( \sup_{0\leq s\leq T} |\widetilde{\mathcal{M}}_s^{N}| \right) +2 \E\left( \sup_{0\leq s\leq T} |J^N(s) | \right)+2 \E\left(\sup_{0\leq s\leq T} K_s^{2,N}\right).
\end{align*}
This together with \eqref{UNS}, Lemmas \ref{Martingale}, \ref{LemporK2}, \ref{Q1+}, \ref{SupKN1S}, Doob's $L^2$-inequality for martingales implies the result.
\end{proof}

We shall need below the 
\begin{lemma}\label{Ssur2}
For any $s>0$,
\begin{eqnarray*}
\int_{0}^{s}\mathbf{1}_{\{V_{r}^{N}=+1\}}dr\longrightarrow \frac{s}{2}; {~~~}\int_{0}^{s}\mathbf{1}_{\{V_{r}^{N}=-1\}}dr\longrightarrow \frac{s}{2}
\end{eqnarray*}
in probability, as $N\longrightarrow \infty$.
\end{lemma}
\begin{proof}
We have (the second line follows from \eqref{DEFHNVN})
\begin{eqnarray*}
\int_{0}^{s}\mathbf{1}_{\{V_{r}^{N}=+1\}}dr+\int_{0}^{s}\mathbf{1}_{\{V_{r}^{N}=-1\}}dr=s,
\end{eqnarray*}
\begin{eqnarray*}
\int_{0}^{s}\mathbf{1}_{\{V_{r}^{N}=+1\}}dr-\int_{0}^{s}\mathbf{1}_{\{V_{r}^{N}=-1\}}dr={(2a_N)}^{-1}H_{s}^{N}.
\end{eqnarray*}
We conclude by adding and substracting the two above identities and using Proposition \ref{supHN}.
\end{proof}

Thus, we have the following result
\begin{lemma}\label{UNSlemConv} 
For any $s>0,$ $J^N(s) \longrightarrow J(s)$ in probability, as $N\rightarrow\infty,$ where $J(s)= (\alpha-\beta)/c.$
\end{lemma}

Moreover, we have the following result which is Proposition 4.23 in  \cite{drame2016non}.
\begin{lemma}\label{M+PhiNs} 
As $N\longrightarrow\infty$, $$\Big( \mathcal{M}_s^{N}, \widetilde{\mathcal{M}}_s^{N} , \ s\geq0\Big)\Longrightarrow \left( \frac{1}{\sqrt{c}} B_{s}^{1}, \frac{1}{\sqrt{c}}  B_{s}^{2}, \ s\geq0\right) \ in \ {(\mathcal{D}([0,\infty)))}^{2},$$ where $B^{1}$ and $B^{2}$ are two mutually independent standard Brownian motions.
\end{lemma}
Let us rewrite \eqref{2DFHN} in the following form 
\begin{align}\label{2DFHN11}
 H_{s}^{N}+\frac{V_{s}^{N}}{2cN}=& \frac{1}{2cN}+G_s^N + \mathcal{M}_s^{N}-\widetilde{\mathcal{M}}_s^{1,N}-\widetilde{\mathcal{M}}_s^{N} + J^N(s)+ \frac{1}{2}(L_{s}^{N}(0)- L_{0^{+}}^{N}(0))
 \end{align}
with
\begin{equation*}
G_s^N= \mathcal{M}_s^{1,N}+ \Phi_s^{1,N} - K_{s}^{2,N}.
\end{equation*}
From \eqref{KNSDEF}, \eqref {M1+} and \eqref{MNa}, we deduce that 
\begin{align}\label{Recall}
G_s^N=&  K_{s}^{1,N} - \frac{2}{c}m_{1,N}\gamma_{1,N}\int_{0}^{s}\mathbf{1}_{\{V_{r}^{N}=-1\}}dr- K_{s}^{2,N}\nonumber\\
=&\frac{1}{cN}\mathcal{P}_{s}^{N,+}+\frac{1}{cN} \int_{0}^{s} \left[\left( \frac{cN}{2}(L_{s}^{N}(H_{r}^{N})-L_{r}^{N}(H_{r}^{N})) \right) \wedge \left({\Lambda}_{\mathcal{P}_{r^{-}}^{N,+}+1}^{1,N} -1\right) \right]d\mathcal{P}_{r}^{N,+}\\&-  \frac{2}{c}m_{1,N}\gamma_{1,N}\int_{0}^{s}\mathbf{1}_{\{V_{r}^{N}=-1\}}dr\nonumber\\
=&G_s^{1,N}+G_s^{2,N}-G_s^{3,N} 
\end{align}
with 
\begin{equation}\label{GN1GN3}
G_s^{1,N}=\frac{1}{cN}\mathcal{P}_{s}^{N,+}, \quad G_s^{3,N}=  \frac{2}{c}m_{1,N}\gamma_{1,N}\int_{0}^{s}\mathbf{1}_{\{V_{r}^{N}=-1\}}dr
\end{equation}
and 
\begin{equation}\label{GN2}
G_s^{2,N}= \frac{1}{cN} \int_{0}^{s} \left[\left( \frac{cN}{2}(L_{s}^{N}(H_{r}^{N})-L_{r}^{N}(H_{r}^{N})) \right) \wedge \left({\Lambda}_{\mathcal{P}_{r^{-}}^{N,+}+1}^{1,N} -1\right) \right]d\mathcal{P}_{r}^{N,+}
\end{equation}
We first prove the 
\begin{lemma}\label{GN1S}
As $N\longrightarrow\infty$, $$\Big(G_s^{1,N}, \ s\geq0\Big)\longrightarrow  0 \ \mbox{ in} \ \mbox{probability}, \ \mbox{locally} \ \mbox{uniformly} \ \mbox{in} \ \mbox{s}.$$ 
\end{lemma}
\begin{proof} 
From \eqref{PN+} we have 
\begin{align*}
G_s^{1,N}=& \frac{1}{cN} \int_{0}^{s}\mathbf{1}_{\{V_{r^{-}}^{N}=-1\}}dP_{r}^{1,N}\\
=&\frac{1}{cN} \int_{0}^{s}\mathbf{1}_{\{V_{r^{-}}^{N}=-1\}}(dP_{r}^{1,N}-2N\gamma_{1,N} dr) +\frac{2}{c}\gamma_{1,N}\int_{0}^{s}\mathbf{1}_{\{V_{r}^{N}=-1\}}dr\\
=&M_s^{\circ,N} +\frac{2}{c}\gamma_{1,N}\int_{0}^{s}\mathbf{1}_{\{V_{r}^{N}=-1\}}dr, {GN1S}
\end{align*}
where $M^{\circ,N}$ is a local martingale (recall that $P^{1,N}$ was a Poisson process with intensity $2N\gamma_{1,N}$). However, we deduce from \ref{GamaN1N2} and the proof of Lemma  \ref{d1Nq01N} that  $$\gamma_{1,N}= d_{1,N}-d_{1,N}q_0^{1,N}= \int_{0}^{\infty}z(1-e^{-Nz})\mu(dz)- \int_{0}^{\infty}z\widetilde{\pi}_{1,N}(dz),$$ recall that $d_{1,N}= \int_{0}^{\infty}z(1-e^{-Nz})\mu(dz).$ It follows from Proposition \ref{PiTildversPiAA} that $\gamma_{1,N} \rightarrow 0$ as $N\rightarrow  \infty.$

However, it is easy to check that 
$$\langle{M^{\circ,N}\rangle}_{s} = \frac{2 \gamma_{1,N}}{c^2N} \int_{0}^{s}\ \mathbf{1}_{\{V_{r}^{N}=-1\}} dr \le\frac{2}{c^2N}s  \int_{0}^{\infty}z \ \mu(dz).$$ Using Doob's inequality, we obtain 
$$\mathbb{E}\left(\sup_{0\leq r\leq s} \left| M_r^{\circ,N}\right| \right)^2 \le C \ \E\left(\langle{M^{\circ,N}\rangle}_{s} \right) \le \frac{Cs}{N},$$
where we have used assumption ${(\bf H)}$ for the last inequality. It follows that $$M_s^{\circ,N} \longrightarrow  0 \ \mbox{ in} \ \mbox{probability}, \ \mbox{locally} \ \mbox{uniformly} \ \mbox{in} \ \mbox{s}.$$ Hence, the desired result follows readily by combining the above arguments. 
\end{proof} 

We shall need below 
\begin{lemma}\label{GN3S}
For any $s>0,$ $G_s^{3,N} \longrightarrow  G_s^3$ in probability, as $N\rightarrow \infty$, where 
$$ G_s^3= \frac{1}{c}s \int_{0}^{\infty}z \ \mu(dz).$$
 \end{lemma}
\begin{proof}  
From \eqref{GN1GN3} and the proof of Lemma \ref{d1Nq01N}, we have 
$$ G_s^{3,N}=  \frac{2}{c}m_{1,N}\gamma_{1,N}\int_{0}^{s}\mathbf{1}_{\{V_{r}^{N}=-1\}}dr= \frac{2}{c}\left( \int_{0}^{\infty}z\widetilde{\pi}_{1,N}(dz) \right) \int_{0}^{s}\mathbf{1}_{\{V_{r}^{N}=-1\}}dr.$$
The result now follows from Proposition \ref{PiTildversPiAA} and Lemma \ref{Ssur2}.
\end{proof} 

 It is easy to obtain from Remarks \ref{MAPPP}, \eqref{Thiabi} and \eqref{GN2} that  
\begin{align}\label{GN2SS}
G_s^{2,N}= \frac{1}{c} \int_{0}^{s}  \int_{0}^{\infty} \mathbf{1}_{\{V_{r^{-}}^{N}=-1\}}  \left(z-\frac{1}{N} \right)  \wedge \left( \frac{c}{2}(L_{s}^{N}(H_{r}^{N})-L_{r}^{N}(H_{r}^{N})) \right) \widetilde{\Pi}^{1,N}(dr,dz). 
\end{align}
Recall that $\widetilde{\Pi}^{1,N}$ is a Poisson random measures on $\mathbb{R}_+^{2}$ with mean measures $2 ds \widetilde{\pi}_{1,N}(dz).$ We now want to estimate the stochastic process $G^{2,N}$. In the next statement, we shall write $\int_a^b$ to mean $\int_{(a,b]},$ except when $b=\infty,$ in which case $\int_{a}^{b}=\int_{(a,\infty)}.$ For $s>0$, define 
\begin{equation*}
Q_{s}^{N}=\frac{1}{c} \int_{0}^{s}  \int_a^b   \left(z-\frac{1}{N} \right)  \wedge \left( \frac{c}{2}(L_{s}^{N}(H_{r}^{N})-L_{r}^{N}(H_{r}^{N})) \right)   \widetilde{\Pi}^{1,N}(dr,dz).
\end{equation*}

Let $\tau$ be a stopping time,  with $\tau \le s$ $a.s.$ We first check that 
\begin{lemma}\label{ParZeng0}
For any $0\le a <b \le \infty,$ we have
\begin{equation*}
\E(Q_{\tau}^{N})= \E (R_{\tau}^{N}),
\end{equation*}
where
\begin{eqnarray*}
R_\tau^N=   \frac{2}{c} \int_{0}^{\tau}   dr \int_a^b  \left(z-\frac{1}{N} \right)  \wedge \left(\frac{c}{2}L_{\tau-r}^{N}(0)-\frac{1}{N} \right) \widetilde\pi_{1,N}(dz).
\end{eqnarray*}
\end{lemma}
\begin{proof} \textit{FISRT STEP} Suppose  $a>0$. We can write the restriction of $ \widetilde\Pi^{1,N}$ to $[0,s]\times[a,+\infty)$ as $ \widetilde\Pi^{1,N}=\sum_{i=1}^{\infty} \delta_{(I_i^N,Z_i^N)}$, where $0<I_1^N<I_2^N\cdots$ are stopping times. Let $\mathcal{F}_s^N=\sigma\{ H_r^N, \ 0\leqslant r \leqslant s\}$. Since $Z_i^N$ is $\mathcal{F}_{I_i^N}^N$-mesurable, we have 
\begin{align*}
\E(Q_{\tau}^{N}) &= \frac{1}{c} \sum_{i=1}^{\infty} \E \left[   \mathbf{1}_{\{{I_i^N}\le \tau,\ a\le Z_i^N\le b\}}\left(Z_i^N-\frac{1}{N} \right)  \wedge \left(\frac{c}{2}(L_{\tau}^{N}(H_{I_i^N}^{N})-L_{I_i^N}^{N}(H_{I_i^N}^{N}))\right)  \right]
\\ &= \frac{1}{c} \sum_{i=1}^{\infty} \E \left[  \E \Big\{ \mathbf{1}_{\{{I_i^N}\le \tau,\ a\le Z_i^N\le b \}}  \left(Z_i^N-\frac{1}{N} \right)  \wedge \left(\frac{c}{2}(L_{\tau}^{N}(H_{I_i^N}^{N})-L_{I_i^N}^{N}(H_{I_i^N}^{N}))\right) \Big|\mathcal{F}_{I_i^N}^N \Big\} \right]
\\ &= \frac{1}{c} \sum_{i=1}^{\infty} \E \left[  \E \Big\{ \mathbf{1}_{\{{I_i^N}\le \tau,\ a\le z\le b \}}  \left(z-\frac{1}{N} \right)  \wedge\left(\frac{c}{2}(L_{\tau}^{N}(H_{I_i^N}^{N})-L_{I_i^N}^{N}(H_{I_i^N}^{N}))\right) \Big|\mathcal{F}_{I_i^N}^N \Big\} \Big| z=Z_i^N \right]
\\ &= \frac{1}{c} \sum_{i=1}^{\infty} \E \left[  \E \Big\{ \mathbf{1}_{\{ r \le \tau,\ a\le z\le b\}}  \left(z-\frac{1}{N} \right)  \wedge\left(\frac{c}{2} (L_{\tau-r}^{N}(0)-L_{0}^{N}(0))\right)  \Big\} \Big| z=Z_i^N,r=I_i^N \right]
\\ &= \frac{1}{c} \E \int_{0}^{\infty} \int_a^b \E \left[  \mathbf{1}_{\{ r \le \tau,\ a\le z\le b\}} \left(z-\frac{1}{N} \right)  \wedge  \left(\frac{c}{2}L_{\tau-r}^{N}(0)-\frac{1}{N}\right) \right]  \widetilde\Pi^{1,N}(dr,dz)
\\ &=  \frac{2}{c} \int_{0}^{\infty} \int_a^b  \E \left[ \mathbf{1}_{\{ r \le \tau,\ a\le z\le b\}}  \left(z-\frac{1}{N} \right)  \wedge \left(\frac{c}{2}L_{\tau-r}^{N}(0)-\frac{1}{N}\right) \right]  dr\widetilde\pi_{1,N}(dz)
\\ &= \frac{2}{c} \E \int_{0}^{\tau}  dr \int_a^b    \left(z-\frac{1}{N} \right)  \wedge\left( \frac{c}{2}L_{\tau-r}^{N}(0)-\frac{1}{N}\right) \widetilde\pi_{1,N}(dz)
\\ &= \E (R_{\tau}^{N}),
\end{align*}
where we have used the fact that for any $u\ge0,$ $L_s^N(u)-L_r^N(u) \stackrel{(d)}{=} L_{s-r}^N(0)-L_{0}^N(0)$ for the 4th equality, the fact that $L_{0}^{N}(0)=L_{0^{+}}^{N}(0)= 2/cN$  for the 5th equality.

\textit{SECOND STEP} : We now treat the case $a=0$. It follows from the above result that for any $k\ge1$, 
\begin{align*}
\frac{1}{c} \E  \int_{0}^{\tau}  \int_{1/k}^b   \left(z-\frac{1}{N} \right)  \wedge \left( \frac{c}{2}(L_{\tau}^{N}(H_{r}^{N})-L_{r}^{N}(H_{r}^{N})) \right)   \widetilde{\Pi}^{1,N}(dr,dz)&=\\  \frac{2}{c} \E \int_{0}^{\tau} dr \int_{1/k}^b  \left(z-\frac{1}{N} \right)  \wedge \left(\frac{c}{2}L_{\tau-r}^{N}(0)-\frac{1}{N} \right) \widetilde\pi_{1,N}(dz). 
 \end{align*}
We can take the limit in that identity as $k\rightarrow \infty$, thanks to the monotone convergence theorem.
\end{proof}

Recall the definition \eqref{GN2SS} of the stochastic process $G^{2,N}$. We have
\begin{lemma}\label{CORR}
For any stopping time $\tau$ such that $\tau \le s$ a.s, $s>0$ arbitrary, 
\begin{equation*} 
\sup_{N\ge1}  \E (G_{\tau}^{2,N}) \le  Cs.
\end{equation*}
where $C$ is an arbitrary constant. 
\end{lemma}
\begin{proof}
From \eqref{GN2SS}, we have that  
\begin{align*}
\E(G_{\tau}^{2,N}) &\le \frac{2}{c} \int_{0}^{s}   dr \int_{\mathbb{R}_+} z \widetilde\pi_{1,N}(dz) \\
 &\le Cs,
 \end{align*}
where we have used assumption ${(\bf H)}$ for the last inequality. 
\end{proof}

We need the following 
\begin{lemma}\label{3.50}
For any $s>0$, there exists $C>0$ such that 
\begin{equation*}
\sup_{N\ge1}\E \left( L_{s}^{N}(0) \right) \le C s.
\end{equation*}
\end{lemma}
\begin{proof}
This is an immediate consequence of \eqref{HsimaN} and the arguments in the proof of Proposition \ref{supHN}. 
\end{proof}

We are now in a position to prove tightness of the stochastic process $G^{2,N}$. To this end, let $\{ \tau_N, \ N\ge 1\}$ be a sequence of stopping times in $[0,s]$. We have the  
\begin{proposition}\label{TightKN2}
For any $s>0$, we have 
\begin{equation*}
 \lim_{t \rightarrow 0} \  \limsup_{N \rightarrow \infty} \ \E \left( G_{\tau_N+t}^{2,N}- G_{\tau_N}^{2,N} \right) =0.
\end{equation*}
\end{proposition}
\begin{proof}
From \eqref{KN2SS}, we have 
\begin{align*}
G_{\tau_N+t}^{2,N} -G_{\tau_N}^{2,N}&= \frac{1}{c} \int_{0}^{\tau_N+t}  \int_{\mathbb{R}_+} \mathbf{1}_{\{V_{r^{-}}^{N}=-1\}}  \left(z-\frac{1}{N} \right)  \wedge \left( \frac{c}{2}(L_{\tau_N+t}^{N}(H_{r}^{N})-L_{r}^{N}(H_{r}^{N})) \right)  \widetilde\Pi^{1,N}(dr,dz)\\
&-\frac{1}{c} \int_{0}^{\tau_N}  \int_{\mathbb{R}_+} \mathbf{1}_{\{V_{r^{-}}^{N}=-1\}}  \left(z-\frac{1}{N} \right)  \wedge \left( \frac{c}{2}(L_{\tau_N}^{N}(H_{r}^{N})-L_{r}^{N}(H_{r}^{N})) \right)   \widetilde\Pi^{1,N}(dr,dz)\\
&=\frac{1}{c} \int_{\tau_N}^{\tau_N+t}  \int_{\mathbb{R}_+} \mathbf{1}_{\{V_{r^{-}}^{N}=-1\}}  \left(z-\frac{1}{N} \right)  \wedge \left( \frac{c}{2}(L_{\tau_N+t}^{N}(H_{r}^{N})-L_{r}^{N}(H_{r}^{N})) \right)  \widetilde\Pi^{1,N}(dr,dz)\\
&+\frac{1}{c} \int_{0}^{\tau_N}  \int_{\mathbb{R}_+} \mathbf{1}_{\{V_{r^{-}}^{N}=-1\}} \bigg[   \left(z-\frac{1}{N} \right)  \wedge \left( \frac{c}{2}(L_{\tau_N+t}^{N}(H_{r}^{N})-L_{r}^{N}(H_{r}^{N})) \right) \\
&-  \left(z-\frac{1}{N} \right)  \wedge \left( \frac{c}{2}(L_{\tau_N}^{N}(H_{r}^{N})-L_{r}^{N}(H_{r}^{N})) \right) \bigg] \widetilde\Pi^{1,N}(dr,dz)\\
&=J_1^{N,\tau_N, t}+ J_2^{N,\tau_N, t}
\end{align*}
It follows that
\begin{equation*}
0 \le G_{\tau_N+t}^{2,N}- G_{\tau_N}^{2,N}  \le J_1^{N,\tau_N, t} + J_2^{N,\tau_N, t}.
\end{equation*}
The Proposition is now a consequence of two next lemmas.
\end{proof}
\begin{lemma}
For any $s>0$, 
\begin{equation*}
 \lim_{t \rightarrow 0} \  \limsup_{N \rightarrow \infty} \ \E \left(J_1^{N,\tau_N, t}\right)=0.
\end{equation*}
\end{lemma}
\begin{proof}
From an adaptation of the argument of Lemma \ref{ParZeng0}, we deduce that 
\begin{align*}
\E \left( J_1^{N,\tau_N, t}\right)&\le \frac{2}{c} \ \E \int_{\tau_N}^{\tau_N+t}  dr \int_{\mathbb{R}_+}  \left(z-\frac{1}{N} \right)  \wedge \left(\frac{c}{2}L_{\tau_N+t-r}^{N}(0)-\frac{1}{N}  \right) \widetilde\pi_{1,N}(dz)\\
&=  \frac{2}{c}\ \E \int_{0}^{t}  dr \int_{\mathbb{R}_+}  z \ \widetilde\pi_{1,N}(dz).
\end{align*}
The rest is entirely similar to the proof of Lemma \ref{CORR}.
\end{proof}

\begin{lemma}
For any $s>0$, 
\begin{equation*}
 \lim_{t \rightarrow 0} \  \limsup_{N \rightarrow \infty} \  \E \left( J_2^{N,\tau_N, t}\right)= 0.
\end{equation*}
\end{lemma}
\begin{proof}
We have 
\begin{align}\label{JN2}
\E \left( J_2^{N,\tau_N, t}\right)  &\le  \frac{1}{c} \ \E \int_{0}^{\tau_N}   \int_{\mathbb{R}_+}   \bigg[    \left(z-\frac{1}{N} \right)  \wedge \left( \frac{c}{2}(L_{\tau_N+t}^{N}(H_{r}^{N})-L_{r}^{N}(H_{r}^{N})) \right) \nonumber\\
&-  \left(z-\frac{1}{N} \right)  \wedge \left( \frac{c}{2}(L_{\tau_N}^{N}(H_{r}^{N})-L_{r}^{N}(H_{r}^{N})) \right) \bigg]  \widetilde\Pi^{1,N}(dr,dz) \nonumber \\
&= \E (J_3^{N,\tau_N, t}).
\end{align}
From an adaptation of the argument of Lemma \ref{ParZeng0}, we deduce that 
\begin{align*}\label{J3N}
\E (J_3^{N,\tau_N, t}) =&  \frac{2}{c} \E \int_{0}^{\tau_N} \! dr \! \int_{\mathbb{R}_+} \!  \bigg[  \!  \left(z-\frac{1}{N} \right) \wedge   \left(\frac{c}{2}L_{\tau_N+t-r}^{N}(0)-\frac{1}{N}\right)\\
&-   \left(z-\frac{1}{N} \right) \wedge   \left(\frac{c}{2}L_{\tau_N-r}^{N}(0)-\frac{1}{N}\right) \! \bigg]  \widetilde\pi_{1,N}(dz)  \nonumber
\\
=&  \frac{2}{c} \ \E \int_{0}^{\tau_N}  dr \! \int_{\mathbb{R}_+} \!  \bigg[  \left(z-\frac{1}{N} \right) \wedge  \left(\frac{c}{2}L_{t+r}^{N}(0)-\frac{1}{N}\right)
\\&- \left(z-\frac{1}{N} \right) \wedge  \left( \frac{c}{2}L_{r}^{N}(0)-\frac{1}{N}\right) \! \bigg] \widetilde\pi_{1,N}(dz) \nonumber
\\
\le&  \frac{2}{c} \ \E \int_{0}^{\tau_N}  dr \! \int_{\mathbb{R}_+} \!  \left(z-\frac{1}{N} \right) \wedge \left[\frac{c}{2} \left(L_{t+r}^{N}(0)- L_{r}^{N}(0) \right) \right]    \widetilde\pi_{1,N}(dz) 
 \nonumber \\
\le&   \frac{2}{c} \int_{0}^{s}  dr \int_{\mathbb{R}_+} \!  \left(z-\frac{1}{N} \right) \wedge \left[\frac{c}{2} \E\Big\{L_{t+r}^{N}(0)- L_{r}^{N}(0) \Big\} \right]   \widetilde\pi_{1,N}(dz) 
\nonumber \\
=&  \frac{2}{c} \int_{0}^{s}  dr \int_{\mathbb{R}_+} \!  \left(z-\frac{1}{N} \right) \wedge \left[\frac{c}{2} \E\Big(L_{t}^{N}(0) \Big) \right]  \widetilde\pi_{1,N}(dz) 
\end{align*}
It follows from \eqref{JN2} and Lemma \ref{3.50}  that
\begin{align*}
\E \left( J_2^{N,\tau_N, t}\right) & \le \frac{2}{c} \int_{0}^{s}  dr \int_{\mathbb{R}_+} \!  \left(z \wedge Ct\right)   \widetilde\pi_{1,N}(dz).  
 \end{align*}
Thanks to Proposition \ref{PiTildversPiAA}, by taking the limit on both side, we then obtain   
\begin{align*}
\limsup_{N \rightarrow \infty} \E \left( J_2^{N,\tau_N, t}\right) & \le \frac{2}{c} \int_{0}^{s}  dr \int_{\mathbb{R}_+} \!  \left(z \wedge Ct\right)   \mu(dz).  
 \end{align*}
Now, from the dominated convergence theorem, we deduce that 
\begin{equation*}
 \lim_{t \rightarrow 0} \  \limsup_{N \rightarrow \infty} \  \E \left( J_2^{N,\tau_N, t}\right)=0,
\end{equation*}
\end{proof}

Thus, Lemma \ref{CORR} combined with Proposition \ref{TightKN2} and Aldous' tightness criterion in \cite{aldous1978stopping} (see e.g. Theorem 16.10 in \cite{patrick1999convergence}) leads to 
\begin{corollary}
The sequence $\{ G^{2,N}, \ N\ge 1\}$ is tight in $\mathcal{D}([0,\infty)).$
\end{corollary}
Recalling \eqref {Recall}, we can rewrite \eqref{2DFHN11} in the form 
\begin{equation}\label{THNGB} 
 H_{s}^{N}= \mathcal{B}_s^N + \frac{1}{2}L_{s}^{N}(0)
\end{equation}
where
\begin{align}\label{MathBNs}
 \mathcal{B}_s^N=  \frac{1}{2cN}-\frac{V_{s}^{N}}{2cN}+  G_s^{1,N}+G_s^{2,N}-G_s^{3,N}+ \mathcal{M}_s^{N}-\widetilde{\mathcal{M}}_s^{1,N}-\widetilde{\mathcal{M}}_s^{N} -\frac{1}{2} L_{0^{+}}^{N}(0)+ J^N(s).
  \end{align}

  We have the following result.
\begin{lemma}
\label{TIGRSs} 
The sequence $\{\mathcal{B}^N, \ N\ge  1\}$ is tight in $\mathcal{D}([0,\infty))$.
\end{lemma}
\begin{proof}
The proof follows by an argument similar to the proof of Lemma 4.37 in  \cite{drame2016non}.
\end{proof}

Recall \eqref{THNGB}. We now deduce the tightness of $H^{N}$ from the above results concerning $ \mathcal{B}^N$, without having to worry about the local time terms.
\begin{proposition}\label{HNTIG}
The sequence $\{H^{N}, \ N\geq1\}$ is tight  in $\mathcal{C}([0,\infty))$.
\end{proposition}
\begin{proof}
The proof follows by an argument similar to the proof of Proposition 4.39 in  \cite{drame2016non}.
\end{proof}

Recall that  $\widetilde{\Pi}^{1,N}$ is a Poisson random measures on $\mathbb{R}_+^{2}$ with mean measures $2 ds \widetilde{\pi}_{1,N}(dz).$ Let \\$D([0,\infty)^2, \mathbb{R}_+)$ denote the space of functions from $[0,\infty)^2$ into $\mathbb{R}_+$ which are right continuous and have left limits at any $(s,z)$ $\in$ $[0,\infty)^2$ (as usual such a function is called c?àdlà?g). We shall always equip the space $D([0,\infty)^2,\mathbb{R}_+)$ with the Skorohod topology. 

The following result is Proposition 2.15 in (\cite{tankov2003financial}, p. 60).
\begin{lemma}\label{PiTild}
The sequence $\{\widetilde{\Pi}^{1,N}, \ N\ge1\}$ converges in distribution in $D([0,\infty)^2, \mathbb{R}_+)$ iff the mean measure $2ds \widetilde{\pi}_{1,N}(dz)$ converges to a measure  $2ds\mu(dz)$. Then $\widetilde{\Pi}^{1,N} \Longrightarrow \widetilde{\Pi},$ in $D([0,\infty)^2, \mathbb{R}_+),$ where $\widetilde{\Pi}$ is a Poisson random measure on $\mathbb{R}_{+}^{2}$ with mean measure $2ds\mu(dz).$
\end{lemma}
Lemma \ref{PiTild} combined with Proposition \ref{PiTildversPiAA} leads to
\begin{corollary}\label{CorImp}
As $N\rightarrow \infty,$ $\widetilde{\Pi}^{1,N} \Longrightarrow \widetilde{\Pi}$ in $D([0,\infty)^2, \mathbb{R}_+),$ where $\widetilde{\Pi}$ is a Poisson random measure on $\mathbb{R}_{+}^{2}$ with mean measure $2ds\mu(dz),$  where $\mu$ was defined in \eqref{R2FINI}.
\end{corollary}
In what follows, we set
\begin{equation}\label{tildmap}
\zeta_r^N=\mathbf{1}_{\{V_{r^{-}}^{N}=-1\}} \quad \mbox{and}  \quad   \widetilde{M}^N \big((0,s]\times(0,z]\big)=\zeta_r^N \widetilde{\Pi}^{1,N}\big((0,s]\times(0,z]\big). 
\end{equation}
We need to prove
\begin{proposition}\label{D20In}
As $N\rightarrow \infty,$ $\widetilde{M}^N \Longrightarrow \Pi$ in $D([0,\infty)^2, \mathbb{R}_+),$ where $\Pi$ is a Poisson random measure on $\mathbb{R}_{+}^{2}$ with mean measure $ds\mu(dz),$ where $\mu$ was defined in \eqref{R2FINI}.
\end{proposition}
We first establish a few lemmas
\begin{lemma}
As $N\rightarrow \infty,$ $\widetilde{M}^N \longrightarrow \Pi$  in the sense of finite-dimensional marginals.
\end{lemma}
\begin{proof}
Let $\bar{s}, \bar{z}>0.$ We have that 
\begin{align*}
\sup_{0\le s\le \bar{s}, \ 0\le z\le \bar{z}} \left\{  \widetilde{M}^N \big((0,s]\times(0,z]\big) \right\} =& \widetilde{M}^N \big((0,\bar{s}]\times(0,\bar{z}]\big)\\
\le& \widetilde{\Pi}^{1,N}\big((0,\bar{s}]\times(0,\bar{z}]\big).
\end{align*}
Hence, from Corollary \ref{CorImp} , it follows easily that for any  $\bar{s}, \bar{z}>0,$ the sequence of random variable $\Big\{ \sup_{0\le s\le \bar{s}, \ 0\le z\le \bar{z}} \widetilde{M}^N \big((0,s]\times(0,z]\big), \ N\ge1 \Big\}$ is tight. This implies immediately that the finite-dimensional marginals are tight. In other words, for all $k\ge1,$ $(s_1,z_1),..., (s_k,z_k),$ the sequence \\$\Big\{ \widetilde{M}^N \big(A_1\big),..., \widetilde{M}^N \big(A_k\big), \ N\ge1 \Big\}$ is tight, with $A_i=(0,s_i]\times(0,z_i],$ for $i=1,...,k.$ Hence at least along a subsequence (but we do not distinguish between the notation for the subsequence and for the sequence), $$ \Big( \widetilde{M}^N \big(A_1\big),..., \widetilde{M}^N \big(A_k\big) \Big) \Longrightarrow \Big( \Pi \big(A_1\big),..., \Pi \big(A_k\big) \Big).$$
\end{proof}

Now we want to identify the limit $\Pi.$ To this end, let us first state a basic result on point process, which will be useful in the sequel. 
\begin{lemma}\label{Pardou}
Assume that there exist a filtration $\{ \mathcal{G}_s, \ s\ge0\}$ such that, for $i=1,...,k,$ the processes $\{  \mathcal{N}_i(s), \ s\ge0\}$ are point processes and $C_i$ are real non negative satisfying: $M_i(s)=\mathcal{N}_i(s)-C_i s$ is a $\mathcal{G}_s$-martingale and $s\rightarrow \sum_{i=1}^{k}  \mathcal{N}_i(s)$ is also a point process. Then for $i=1,...,k,$ the processes $\{ \mathcal{N}_i(s), \ s\ge0\}$ are mutually independent Poisson processes, with respective intensities $C_i.$
\end{lemma}
\begin{proof}
It is enough to show that for any $0<r<s,$ for all $\alpha_i$ $\in$ $\mathbb{R},$ $i=1,...,k,$
$$\E^{\mathcal{G}_r} \exp \Big\{ -\sum_{i=1}^k \alpha_i \left( \mathcal{N}_i(s)-\mathcal{N}_i(r)\right) \Big\}= \prod_{i=1}^k  \exp \left\{-C_i (s-r)(1-e^{-\alpha_i}) \right\}.$$
In the following calculation, we will exploit the fact that $\mathcal {N} _1 (s),...,\mathcal {N} _k (s)$ never jump at the same time, which follows from the fact that $s\rightarrow \sum_{i=1}^{k}  \mathcal{N}_i(s)$ is a point process. $$ \exp \Big(- \sum_{i=1}^k \alpha_i \mathcal{N}_i(s) \Big)=\exp \Big(- \sum_{i=1}^k \alpha_i \mathcal{N}_i(r) \Big)+ \sum_{j=1}^k \int_{(r,s]} (e^{-\alpha_j}-1)e^{- \sum_{i} \alpha_i \mathcal{N}_i(u^-)} d\mathcal{N}_j(u).$$ Hence 
\begin{align*}
\E^{\mathcal{G}_r} \Big[ \exp \Big(- \sum_{i=1}^k \alpha_i \mathcal{N}_i(s) \Big)- \exp \Big(- \sum_{i=1}^k \alpha_i \mathcal{N}_i(r) \Big) \Big]=& \sum_{j=1}^k \E^{\mathcal{G}_r}  \int_{(r,s]} (e^{-\alpha_i}-1)e^{- \sum_{i} \alpha_i \mathcal{N}_i(u^-)} d\mathcal{N}_j(u)\\
 =& \sum_{j=1}^k C_j \E^{\mathcal{G}_r}  \int_{(r,s]} (e^{-\alpha_j}-1)e^{- \sum_{i} \alpha_i \mathcal{N}_i(u^-)} du.
\end{align*}
Consequently
$$\E^{\mathcal{G}_r} \exp \Big\{- \sum_{i=1}^k \alpha_i \left( \mathcal{N}_i(s)-\mathcal{N}_i(r)\right) \Big\}=  \exp \Big\{-(s-r)\sum_{i=1}^k C_i(1-e^{-\alpha_i}) \Big\}.$$ The desired result follows.
\end{proof}
 
$\bf{Proof \ of \  Proposition}$ \ref{D20In}. Now, from the convergence of finite-dimensional marginals, we deduce that for any $k\ge1,$ $0<z_1<...<z_k,$
\begin{align}\label{Margfini}
\Big( \widetilde{M}^N\big((0,s]\times(0,z_1]\big),\widetilde{M}^N\big((0,s]\times(z_1,z_2]\big),..., \widetilde{M}^N\big((0,s]\times(z_{k-1},z_k]\big)\Big) \Longrightarrow \nonumber \\\Big( \Pi \big((0,s]\times(0,z_1]\big), \Pi \big((0,s]\times(z_1,z_2]\big),..., \Pi \big((0,s]\times(z_{k-1},z_k]\big) \Big), \quad \mbox{as} \quad N\rightarrow \infty.
\end{align}
Let us define
$\widehat{\mathcal{M}}^N(s)=\left( \begin{array}{ccc} & \widetilde{M}^N\big((0,s]\times(0,z_1]\big)&\\&\cdot& \\&\cdot&\\&\cdot& \\&\widetilde{M}^N\big((0,s]\times(z_{k-1},z_k]\big)& \end{array}\right). $

Note that $\widehat{\mathcal{M}}^N(s)$ [resp. $ \widetilde{M}^N\big((0,s]\times(0,z_k]\big)$] is a point process with value in $\mathbb{R}^k$ (resp. $\mathbb{R}$). However, let $r,T>0.$ For $\ell\ge1,$ let $\phi$ $\in$ $\mathcal{C}_b(\mathbb{R}^\ell,\mathbb{R})$ and let $f_1,...,f_{\ell}$ be a sequence of function whose support is included in $(0,r]\times(0,T].$ It is easily seen that 
$$\widetilde{M}^N\big((r,s]\times(z_{i-1},z_i]\big)- 2 \widetilde{\pi}_{1,N}(z_{i-1},z_i]\int_{r}^{s} \zeta_u^Ndu$$ is a martingale (see e.g. chap 6 in \cite{ccinlar2011probability}). It is also plain that 
$$ \E \bigg\{\left(\widetilde{M}^N\big((r,s]\times(z_{i-1},z_i]\big)- 2 \widetilde{\pi}_{1,N}(z_{i-1},z_i]\int_{r}^{s} \zeta_u^Ndu \right)\times \phi \left(\widetilde{M}^N(f_1),..., \widetilde{M}^N(f_{\ell}) \right) \bigg\}=0.$$
It follows from Proposition \ref{PiTildversPiAA}, Lemma \ref{Ssur2} and \eqref{Margfini} that 
\begin{align*}
\lim_{N\rightarrow\infty} &\E \bigg\{\left(\widetilde{M}^N\big((r,s]\times(z_{i-1},z_i]\big)- 2 \widetilde{\pi}_{1,N}(z_{i-1},z_i]\int_{r}^{s} \zeta_u^Ndu \right)\times \phi \left(\widetilde{M}^N(f_1),..., \widetilde{M}^N(f_{\ell}) \right) \bigg\}\\
&= \E \bigg\{\left(\Pi \big((r,s]\times(z_{i-1},z_i]\big)- \mu(z_{i-1},z_i] (s-r) \right)\times \phi \left(\Pi(f_1),..., \Pi(f_{\ell}) \right) \bigg\}=0.
\end{align*}
Now let $\mathcal{G}$ be the filtration defined by 
$$\mathcal{G}_r= \sigma \bigg\{\phi \left(\Pi(f_1),..., \Pi(f_{\ell}) \right), \ \forall \ \ell \ge1, \ \forall \ \phi \ \in \mathcal{C}_b(\mathbb{R}^\ell,\mathbb{R}), \ \forall  (f_1,...,f_{\ell}) \subset \ (0,r]\times(0,T] \bigg\}.$$
 Hence, it is easy to check that for any $2\le i \le k,$ $\Pi \big((0,s]\times(z_{i-1},z_i]\big)- s\mu(z_{i-1},z_i]$ is a $\mathcal{G}_s$-martingale. Moreover, $\left\{ \Pi \big((0,s]\times(z_{i-1},z_i]\big) \right\}_{2\le i\le k}$ are point processes and $ \Pi \big((0,s]\times(z_{1},z_k]\big)$ is also a point process. We deduce from Lemma \ref{Pardou} that for $i=1,...,k,$ the processes $ \Pi \big((0,s]\times(z_{i-1},z_i]\big)$ are mutually independent Poisson processes, with respective intensities $\mu(z_{i-1},z_i].$ Summarizing, we obtain that $\Pi$ is a point process on $(0,\infty)^2$ such that for any $k>0,$  $0<z_1<...<z_k,$ $\Pi \big((0,s]\times(z_1,z_2]\big),..., \Pi \big((0,s]\times(z_{k-1},z_k]\big)$ are mutually independent Poisson process, with respective intensities $\mu(z_1,z_2],...,\mu(z_{k-1},z_k].$ Consequently, $\Pi$ is a Poisson random measure on $\mathbb{R}_+^2$ with mean measure $ds \mu(dz).$
 
 Now, it remains to show the functional convergence of $\widetilde{M}^N.$ To this end, combining the definition of Bickel and Wichura in \cite{bickel1971convergence}, p. 1663 of $w_\delta^{\prime\prime},$ the modulus of continuity and Corollary \ref{CorImp}, we deduce that $\widetilde{\Pi}^{1,N}$ satisfy condition (10) of Corollary in \cite{bickel1971convergence}. Moreover, since  $w_\delta^{\prime\prime}(\widetilde{M}^N) \le w_\delta^{\prime\prime}(\widetilde{\Pi}^{1,N}),$ then $\widetilde{M}^N$ satisfy the same condition (10). Hence the desired result follows by combining the above results with Corollary in \cite{bickel1971convergence}. $\hfill \blacksquare$
 
\begin{corollary}
We have moreover  
$$\begin{pmatrix} 
\mathcal{M}_s^{N} -\widetilde{\mathcal{M}}_s^{N} \\\\
\widetilde{M}^N
\end{pmatrix} \Longrightarrow \begin{pmatrix} 
\sqrt{\frac{2}{c}}B_s \\\\
\Pi
\end{pmatrix}$$
as $N\rightarrow \infty,$ Where $(B_s, s\ge0)$ is a standard brownian motion independent of the Poisson random measure $\Pi$.
\end{corollary}
\begin{proof}
The fact that the vector converges weakly along a subsequence follows from tightness. We have identified the limit of the first (resp. of the second) coordinate in Lemma \ref{M+PhiNs} (resp. in Proposition \ref{D20In}). The fact that the two components of the limit are independent follows from an easy extension of Lemma \ref{Pardou}. Finally the whole sequence converges, since the limit is unique.   
\end{proof}

Recall \eqref{CHS1}. We have 
\begin{equation*}
c H_s= Y_s - \inf_{0\leqslant r \leqslant s} Y_r - \int_{0}^{s}\int_{0}^{\infty}\left(z + \inf_{r\leqslant u \leqslant s}(Y_u-Y_r) \right)^+\Pi(dr, dz),
\end{equation*}
where
$$Y_s= - b s + \sqrt{2c} B_s + \int_{0}^{s}\int_{0}^{\infty}z \overline{\Pi}(dr, dz).$$
Let $Y_s^k$ be the Lévy process defined by 
\begin{equation*}
Y_s^k=- b s + \sqrt{2c} B_s + \int_{0}^{s}\int_{1/k}^{\infty}z \overline{\Pi}(dr, dz)
\end{equation*}
Let $H_s^k$ be the height process associated to the Lévy process $Y_s^k$. In other words, $H_s^k$ is 
given by 
\begin{equation}\label{cHsk}
c H_s^k= Y_s^{k} - \inf_{0\leqslant r \leqslant s} Y_r^k + \int_{0}^{s}\int_{1/k}^{\infty}\left(z + \inf_{r\leqslant u \leqslant s}(Y_u^k-Y_r^k) \right)^+  \Pi(dr, dz)
\end{equation}
We first prove
\begin{lemma}\label{Lemint}
As $k\rightarrow \infty$, $Y_s^k \rightarrow Y_s$ in mean square, locally uniformly with respect to $s$.
\end{lemma}
\begin{proof}
It is plain that 
\begin{align*}
\E\left(\sup_{0\le r\le s}|Y_s-Y_s^k|^2\right) &\le \E\left[\sup_{0\le r\le s} \left|\int_{0}^{r}\int_{0}^{1/k}z \overline{\Pi}(dr, dz) \right|^2\right] \\
&\le 4s \int_{0}^{1/k}z^2 \mu(dz)\\
&\longrightarrow 0, \quad \mbox{as} \quad  k\rightarrow \infty,
\end{align*}
where we have used Doob's inequality. The result follows.
\end{proof}

We shall need below
\begin{proposition}\label{4.48}
For any $s>0$, $H_s^k \rightarrow H_s$ in probability, locally uniformly in $s$.
\end{proposition}
\begin{proof}
From Lemma \ref{Lemint} follows that $$Y_s^{k} - \inf_{0\leqslant r \leqslant s} Y_r^k \longrightarrow Y_s - \inf_{0\leqslant r \leqslant s} Y_r$$ in mean square, locally uniformly in $s$. We now consider the last term in \eqref{cHsk} and prove pointwise convergence. We first notice that 
\begin{align*}
 \E \int_{0}^{s}\int_{0}^{1/k} \left(z + \inf_{r\leqslant u \leqslant s}(Y_u-Y_r) \right)^+\Pi(dr, dz) \le& s \int_{0}^{1/k}z \mu(dz)\\
&\longrightarrow 0, \quad \mbox{as} \quad  k\rightarrow \infty.
\end{align*}
From an adaptation of the argument of Lemma \ref{ParZeng0}, we deduce that 
\begin{align*}
 &\E \int_{0}^{s}\int_{1/k}^{\infty}  \bigg| \left(z + \inf_{r\leqslant u \leqslant s}(Y_u-Y_r) \right)^+ -\left(z + \inf_{r\leqslant u \leqslant s}(Y_u^k-Y_r^k) \right)^+ \bigg|\Pi(dr, dz) \\&=
 \E \int_{0}^{s}dr \int_{1/k}^{\infty}  \bigg| \left(z + \inf_{0\leqslant r \leqslant s}Y_r \right)^+ -\left(z + \inf_{0\leqslant r \leqslant s} Y_r^k \right)^+  \bigg|\mu(dz)\\
 &\le  \E \int_{0}^{s}dr \int_{1/k}^{\infty}  z\wedge \bigg| \inf_{0\leqslant r \leqslant s}Y_r-\inf_{0\leqslant r \leqslant s} Y_r^k \bigg|\mu(dz)\\
&\le  \int_{0}^{s}dr \int_{1/k}^{\infty}  z\wedge \E \left\{ \left| \inf_{0\leqslant r \leqslant s}Y_r-\inf_{0\leqslant r \leqslant s} Y_r^k \right| \right\}\mu(dz)
\end{align*}
We deduce from Lemma \ref{Lemint} that 
$$  \E \left| \inf_{0\leqslant r \leqslant s}Y_r-\inf_{0\leqslant r \leqslant s} Y_r^k \right| \longrightarrow 0, \quad \mbox{as} \quad  k\rightarrow \infty.$$
The desired result follows from the dominated convergence theorem.
\end{proof}

From \eqref{GN2SS}, we have  
\begin{align*}
G_s^{2,N}= \frac{1}{c} \int_{0}^{s}  \int_{0}^{\infty} \left(z-\frac{1}{N} \right) \widetilde{M}^N(dr,dz)-   \frac{1}{c} \int_{0}^{s}  \int_{0}^{\infty} \left(z-\frac{1}{N} - \frac{c}{2}(L_{s}^{N}(H_{r}^{N})-L_{r}^{N}(H_{r}^{N})) \right)^+ \widetilde{M}^N(dr,dz). 
\end{align*}
Let us now rewrite \eqref{THNGB}  in the form 
\begin{equation}\label{final}
 cH_{s}^{N}= Y_s^N + \frac{1}{2}L_{s}^{N}(0) - \int_{0}^{s}  \int_{0}^{\infty} \left(z-\frac{1}{N} - \frac{c}{2}(L_{s}^{N}(H_{r}^{N})-L_{r}^{N}(H_{r}^{N})) \right)^+ \widetilde{M}^N(dr,dz)
\end{equation}
with
$$Y_s^N= \varepsilon_s^N+ cJ^N(s)+ B_s^N+\int_{0}^{s}  \int_{0}^{\infty} \left(z-\frac{1}{N} \right) \widetilde{M}^N(dr,dz)-2\left( \int_{0}^{\infty}z\widetilde{\pi}_{1,N}(dz) \right) \int_{0}^{s}\mathbf{1}_{\{V_{r}^{N}=-1\}}dr $$
and
$$  \varepsilon_s^N=\frac{1}{2N}-\frac{V_{s}^{N}}{2N}+ cG_s^{1,N} -c\widetilde{\mathcal{M}}_s^{1,N} -\frac{c}{2} L_{0^{+}}^{N}(0) \quad \mbox{and} \quad  B_s^N=c\left(\mathcal{M}_s^{N}-\widetilde{\mathcal{M}}_s^{N}\right).$$

By combining Corollary \ref{ConvM} with Lemma \ref{GN1S}, we have prove 
\begin{lemma}\label{EPLN}
As $N\longrightarrow\infty,$ $\varepsilon_s^N\longrightarrow  0 \ \mbox{ in} \ \mbox{probability}, \ \mbox{locally} \ \mbox{uniformly} \ \mbox{in} \ \mbox{s}.$
\end{lemma}

Let us define $$\varepsilon_s^{k,N}=\frac{1}{2N}-\frac{V_{s}^{k,N}}{2N}+ cG_s^{k,1,N} -c\widetilde{\mathcal{M}}_s^{k,1,N} -\frac{c}{2} L_{0^{+}}^{N}(0), $$
$$Y_s^{k,N}= \varepsilon_s^{k,N}+ cJ^N(s)+ B_s^N+\int_{0}^{s}  \int_{1/k}^{\infty} \left(z-\frac{1}{N} \right) \widetilde{M}^N(dr,dz)-2\left( \int_{1/k}^{\infty}z\widetilde{\pi}_{1,N}(dz) \right) \int_{0}^{s}\mathbf{1}_{\{V_{r}^{k,N}=-1\}}dr$$ and 
\begin{equation}\label{CHNKS1}
 cH_{s}^{k,N}= Y_s^{k,N} + \frac{c}{2}L_{s}^{k,N}(0) +  \int_{0}^{s}  \int_{1/k}^{\infty} \left(z-\frac{1}{N} - \frac{c}{2}(L_{s}^{k,N}(H_{r}^{k,N})-L_{r}^{k,N}(H_{r}^{k,N})) \right)^+ \widetilde{M}^N(dr,dz)
\end{equation}
We first prove 
\begin{lemma}\label{GNkMNk}
For any $s>0$, as $N\rightarrow \infty$,
$$\sup_{k\ge1}\E\left(\sup_{0\le r\le s}|G_s^{k,1,N}|\right) + \sup_{k\ge1}\E\left(\sup_{0\le r\le s}|\widetilde{\mathcal{M}}_s^{k,1,N}|\right) \longrightarrow 0.$$

\end{lemma}
\begin{proof}
From the proof of Lemma \ref{GN1S}, we have 
$$G_s^{k,1,N}=M_s^{k,\circ,N} +\frac{2}{c}\gamma_{k,1,N}\int_{0}^{s}\mathbf{1}_{\{V_{r}^{N}=-1\}}dr,$$ where 
$$\langle{M^{k,\circ,N}\rangle}_{s} \le\frac{2}{c^2N}s \int_{1/k}^{\infty}z \ \mu(dz) \quad \mbox{and} \quad \gamma_{k,1,N}=\int_{1/k}^{\infty}z(1-e^{-Nz})\mu(dz)- \int_{1/k}^{\infty}z\widetilde{\pi}_{1,N}(dz).$$  
Hence using Jensen's and Doob's inequalities, it is plain that 
$$\E\left(\sup_{0\le r\le s}|G_s^{k,1,N}|\right) \le \frac{Cs}{N}+\int_{1/k}^{\infty}z(1-e^{-Nz})\mu(dz)- \int_{1/k}^{\infty}z\widetilde{\pi}_{1,N}(dz).$$ The first assertion follows easily from Proposition \eqref{PiTildversPiAA}. However, from the proof of Lemma \ref{Martingale}, we can prove similarly the second assertion. 
\end{proof}

We need to prove 
\begin{lemma}\label{EPSILONk}
For any $s>0$, as $N\rightarrow \infty$,
$$\sup_{k\ge1}\E\left(\sup_{0\le r\le s}\left\{ |\varepsilon_r^{N}|+|\varepsilon_r^{k,N}|\right\} \right) \longrightarrow 0.$$
 
\end{lemma}
\begin{proof}
We have 
$$\sup_{0\le r\le s}|\varepsilon_r^{k,N}| \le c \left( \sup_{0\le r\le s}|G_s^{k,1,N}| + \sup_{0\le r\le s}|\widetilde{\mathcal{M}}_s^{k,1,N}| \right).$$ The result now follows by combining the above arguments with Lemmas  \ref{EPLN} and \ref{GNkMNk}.
\end{proof}

We shall need 
\begin{lemma}\label{Lem2int1}
For any $s>0$, there exists a function $C:\Z_+\mapsto\R_+$ such that $C(n)\to0$ as $n\to\infty$
and
$$\E\left(\sup_{0\le r\le s}|Y_r^N-Y_r^{k,N}|\right) \le C(N)+ C(k). $$
\end{lemma}
\begin{proof}
It is plain that 
$$\sup_{0\le r\le s}|Y_r^N-Y_r^{k,N}| \le \sup_{0\le r\le s}|\varepsilon_r^{N}-\varepsilon_r^{k,N}|+\int_{0}^{s}  \int_{0}^{1/k} \left(z-\frac{1}{N} \right) \widetilde{M}^N(dr,dz)+ 2s \int_{0}^{1/k}z\widetilde{\pi}_{1,N}(dz).$$ It follows that
\begin{align*}
\E\left(\sup_{0\le r\le s}|Y_r^N-Y_r^{k,N}|\right) \le \E\left(\sup_{0\le r\le s}\left\{ |\varepsilon_r^{N}|+|\varepsilon_r^{k,N}|\right\}\right) + 4 s \int_{0}^{1/k}z \mu(dz)
\end{align*}
The desired result follows by combining this with Lemma \ref{EPSILONk}. 
\end{proof}

We prove the
\begin{lemma}\label{Lem2int}
For any $s>0$, there exists a function $C:\Z_+\mapsto\R_+$ such that $C(n)\to0$ as $n\to\infty$
and
$$ \E\left(\sup_{0\le r\le s}|L_r^N(0)-L_r^{k,N}(0)|\right) \le C(N)+C(k). 
$$
\end{lemma}
\begin{proof}
We first note that, whenever $H_s^N$ (resp. $H_s^{k,N}$) hits $0$, the last term in the corresponding formula \eqref{final} for $H_s^N$  (resp. \eqref{CHNKS1} for $H_s^{k,N}$) equals zero, so that at each time  $s$ where $L_s^N(0)$ (resp. $L_s^{k,N}(0)$) increases, 
$$ cH_s^N= Y_s^N+ \frac{c}{2}L_s^N(0), \quad \mbox{resp.} \quad cH_s^{k,N}= Y_s^{k,N}+ \frac{c}{2}L_s^{k,N}(0).$$ The result will follow from Lemma \ref{Lem2int1}, and the fact that as $N\rightarrow \infty$,
\begin{align*}
\E&\left(\sup_{0\le r\le s} \left|\frac{c}{2}L_s^N(0)+ \inf_{0\le u\le r} Y_u^N\right | \right)\longrightarrow 0,\\
\sup_{k\ge1}\E&\left(\sup_{0\le r\le s} \left|\frac{c}{2}L_s^{k,N}(0)+ \inf_{0\le u\le r} Y_u^{k,N}\right |\right) \longrightarrow 0.
\end{align*}
We will establish only the second statement. At each time $s$ when  $H_s^{k,N}$ hits zero, $cH_s^{k,N}+ \frac{V_s^{k,N}}{2N}$ makes a jump of size $1/N$, which is the increase in  $\frac{c}{2}L_s^{k,N}(0)$. However, when $H_s^{k,N}$ hits zero,  $Y_s^{k,N}$ continues to go down to a negative value which, due to the properties of the Poisson processes, has an exponential law with the parameter $N + \rho(k,N)$, where $$\sup_{k\ge1} \frac{\rho(k,N)}{N}\longrightarrow 0,\quad \mbox{as} \ N \rightarrow \infty.$$ Consequently,  each visit of $H_s^{k,N}$ to zero yields an increase in $-\inf_{0\le r\le s}Y_r^{k,N}$  which is an exponential random variable with the parameter $N+\rho(k,N)$. The r.v.'s corresponding to successive visits of $H_s^{k,N}$ to zero
are mutually independent. The number of those up to time $s$ is of the order of $N$, since of $\E L_s^{k,N}(0)\le Cste$. Hence it follows from a martingale argument that, if $\xi_i$ denote mutually independent Exp($N+\rho(k,N)$) random variables,   
\begin{align*}
\E&\left(\sup_{0\le r\le s} \left|\frac{c}{2}L_s^{k,N}(0)+ \inf_{0\le u\le r} Y_u^{k,N}\right|^2 \right)   \\
&\le \E \left\{\left| \sum_{i=1}^{CN} (\xi_i-\frac{1}{N})\right|^2 \right\}\\
&\le CN Var (\xi_1)+ C^2N^2 \left| \frac{1}{N+\rho(k,N)}-\frac{1}{N} \right|^2\\
&\le \frac{C}{N}+ C^2 \left( \frac{\rho(k,N)}{N}\right)^2\\
&\rightarrow 0, 
\end{align*}
as $N \rightarrow \infty$, uniformly w.r.t. $k$. 
\end{proof}

We have also
\begin{proposition}\label{4.53}
For any $s>0$, there exists a function $C:\Z_+\mapsto\R_+$ such that $C(n)\to0$ as $n\to\infty$
and 
$$\E\left(\sup_{0\le r\le s}|H_r^N-H_r^{k,N}|\right) \le C(N)+ C(k).$$
\end{proposition}
\begin{proof}
We now consider the last term in \eqref{CHNKS1} and prove pointwise convergence. We first notice that 
\begin{align*}
\E \int_{0}^{s}\int_{0}^{1/k} \left(z-\frac{1}{N} -\frac{c}{2}(L_{s}^{N}(H_{r}^{N})-L_{r}^{N}(H_{r}^{N})) \right)^+ \widetilde{M}^N(dr,dz) &\le 2s\int_{0}^{\infty}z\widetilde{\pi}_{1,N}(dz)\\&\le 2s \int_{0}^{1/k}z \mu(dz)\\
&\longrightarrow 0, \quad \mbox{as} \quad  k\rightarrow \infty.
\end{align*}
For the rest of this proof, we set $$\Delta_{s,r}^N=\frac{c}{2}(L_{s}^{N}(H_{r}^{N})-L_{r}^{N}(H_{r}^{N})) \quad \mbox{and} \quad \Delta_{s,r}^{k,N}= \frac{c}{2}(L_{s}^{k,N}(H_{r}^{k,N})-L_{r}^{k,N}(H_{r}^{k,N})).$$ From an adaptation of the argument of Lemma \ref{ParZeng0}, we deduce that 
\begin{align*}
 &\E \int_{0}^{s}\int_{1/k}^{\infty}  \bigg| \left(z-\frac{1}{N} - \Delta_{s,r}^N \right)^+- \left(z-\frac{1}{N} - \Delta_{s,r}^{k,N} \right)^+  \bigg|\widetilde{M}^N(dr,dz)  \\
&\le\E \int_{0}^{s}\int_{1/k}^{\infty}  \bigg| \left(z-\frac{1}{N} - \Delta_{s,r}^{N}\right)^+- \left(z-\frac{1}{N} - \Delta_{s,r}^{k,N} \right)^+  \bigg|\widetilde{\Pi}^{1,N}(dr,dz)
 \\&=
2 \E \int_{0}^{s}dr \int_{1/k}^{\infty}  \bigg|  \left(z- \frac{c}{2}L_{r}^{N}(0)\right)^+ - \left(z- \frac{c}{2}L_{r}^{k,N}(0)\right)^+ \bigg|\widetilde\pi_{1,N}(dz) 
 \\
&\le  2 \E  \int_{0}^{s}dr \int_{1/k}^{\infty}  z \wedge \left[\frac{c}{2} \Big|L_{
r}^{N}(0)- L_{r}^{k,N}(0) \Big| \right]   \widetilde\pi_{1,N}(dz) 
 \nonumber \\
&\le 2 \E \int_{0}^{s}dr \int_{1/k}^{\infty}  \! z\wedge \left[\frac{c}{2} \E\Big| L_{
r}^{N}(0)- L_{r}^{k,N}(0) \Big| \right]   \widetilde\pi_{1,N}(dz). 
\end{align*}
We deduce from Lemma \ref{Lem2int} that 
$$ \E\left(|L_r^N(0)-L_r^{k,N}(0)|\right) \longrightarrow 0, \quad \mbox{as} \quad N \mbox{ and }  k\rightarrow \infty.$$
The desired result follows from the dominated convergence theorem.
\end{proof}

We can establish 
\begin{lemma}\label{Lint}
For all $k\ge1$, $H^{k,N}\Longrightarrow H^k$ in $\mathcal{C}([0,\infty))$ as $N\rightarrow \infty.$
\end{lemma}
\begin{proof}
We first notice that in this situation $\mu$ is a finite measure. We can rewrite $ \widetilde{M}^{N}=\sum_{i=1}^{\infty} \delta_{(I_i^N,Z_i^N)}$ (resp. $ \Pi=\sum_{i=1}^{\infty} \delta_{(I_i,Z_i)}$), where $I_1^N,I_2^N,\cdots$ (resp. $I_1,I_2,\cdots$ ) are stopping times. Let us rewrite \eqref{cHsk} (resp. \eqref{CHNKS1}) in the form 
\begin{equation*}
cH_s^k= Y_s^{k,c} - \inf_{0\leqslant r \leqslant s} Y_r^k +  \frac{1}{c} \sum_{i=1}^{\infty}  \mathbf{1}_{\{{I_i}\le s,\ Z_i \ge \frac{1}{k}\}} Z_i  \wedge \left(- \inf_{I_i \leqslant u \leqslant s}(Y_u^k-Y_{I_i^N}^k) \right) 
\end{equation*}
resp.
\begin{equation*}
 H_{s}^{k,N}= Y_s^{k,N} + \frac{1}{2}L_{s}^{k,N}(0) + \frac{1}{c} \sum_{i=1}^{\infty} \mathbf{1}_{\{{I_i^N}\le s,\ Z_i^N\ge \frac{1}{k}\}}\left(Z_i^N-\frac{1}{N} \right)  \wedge \left(\frac{c}{2}(L_{s}^{k,N}(H_{I_i^N}^{k,N})-L_{I_i^N}^{k,N}(H_{I_i^N}^{k,N}))\right)  
\end{equation*}
Until the first jump, we have exactly the same situation as that in Dramé et all in \cite{drame2016non}, Proposition 4.40. So in particular $L_s^{k,N}(0)\rightarrow L_s^k(0)=-\frac{2}{c}\inf_{0\leqslant r \leqslant s} Y_r^k$. From the first jump $(I_1^N,Z_1^N)$, the process $H_{s}^{k,N}$ is reflected above the level  $H_{I_1^N}^{k,N}$ until the local time associated with this level accumulated from time $I_1^N$ reaches the level $Z_1^N$. Again the approximate equation is the same as in \cite{drame2016non}, with the level $0$  moved to the level $H_{I_1^N}^{k,N}$. The proof follows the same reasoning. The proof of the lemma follows by repeating this same argument. 
\end{proof}

We are now ready to state the main result.
\begin{theorem}\label{THGama}
As $N \longrightarrow\infty$,
\begin{align*}
\bigg(H^{N}, Y^N \bigg)&\Rightarrow \bigg(H, Y\bigg) \ in \ {(\mathcal{C}([0,\infty)))}^{2}
\end{align*}
as $N\rightarrow \infty$, where $H$ is the unique weak solution of the SDE 
\begin{equation*}
c H_s= Y_s - \inf_{0\leqslant r \leqslant s} Y_r - \int_{0}^{s}\int_{0}^{\infty}\left(z + \inf_{r\leqslant u \leqslant s}(Y_u-Y_r) \right)^+\Pi(dr, dz),
\end{equation*}
 where $Y$ was defined in \eqref{YL\'{e}vy}.
\end{theorem}
\begin{proof}
We only prove that $H^N\Longrightarrow H$, in order to simplify our notations. The proof for the pair $(H^N, Y^N)$ follows the exact same reasoning. It follows from tightness that along a subsequence,
$$\begin{pmatrix} 
H^{N}  \\\\
{H}^{k,N}
\end{pmatrix} \Longrightarrow \begin{pmatrix} 
H^{\prime} \\\\
H^k
\end{pmatrix}$$
as $N\rightarrow \infty.$ All we want to show is that $H=H^{\prime}$. This will follow from a combination of Proposition \ref{4.48}, Proposition \ref{4.53} and Lemma \ref{Lint}. Indeed, from Proposition \ref{4.53}, for any $\delta>0$, we can choose $N_{\delta}$ and $k_{\delta}$ such that for all $N\ge N_{\delta}$, $k\ge k_{\delta}$, $$\P \left( \sup_{0\le r\le s}|H_r^N-H_r^{k,N}| > \delta\right)\le \delta.$$ Consequently, if $\phi$ $\in$ $\mathcal{C}_b(\mathbb{R}_+, [0,1])$ satisfies  $\mathbf{1}_{[ 0, \delta] } (x) \le \phi(x) \le \mathbf{1}_{[ 0, 2\delta] } (x)$,  it follows from Lemma \ref{Lint}
$$\E\left[ \phi \left(\sup_{0\le r\le s}|H_r^{\prime}-H_r^{k}|\right)\right] = \lim_{N\rightarrow \infty} \E\left[ \phi \left(\sup_{0\le r\le s}|H_r^N-H_r^{k,N}|\right)\right] \ge 1-\delta.$$ Consequently, whenever $k\ge k_{\delta}$, 
$$\P \left( \sup_{0\le r\le s}|H_r^{\prime}-H_r^{k}| \le 2\delta\right)\ge1- \delta.$$ Hence Proposition \ref{4.48}, $H_r^{\prime}=H_r.$
\end{proof}
\begin{remark}
From our convergence results, we can as in \cite{schneider2003forcing} deduce the well known second Ray-Knight theorem, in the subcritical and critical cases (i.e $\alpha\le \beta).$
\end{remark}

\frenchspacing
\bibliographystyle{plain}

\end{document}